\documentclass[11pt]{amsart}
\usepackage{amssymb}
\usepackage{amscd}
\usepackage{amsfonts}
\usepackage{amsmath}
\usepackage{verbatim}
\usepackage{wasysym}
\usepackage[all]{xy}

\newcommand{\Span}{\hbox{\rm span}}
\newcommand{\Hol}{\hbox{\sl Hol}}

\theoremstyle{plain}
\newtheorem{definition}{Definition}[section]
\newtheorem{lemma}[definition]{Lemma}
\newtheorem{thm}[definition]{Theorem}
\newtheorem{prop}[definition]{Proposition}
\newtheorem{cor}[definition]{Corollary}
\newtheorem{rem}[definition]{Remark}
 \def\ga{{\mathfrak a}}

 \def\ggg{{\mathfrak g}} 
 \def\gh{{\mathfrak h}}

 \def\gm{{\mathfrak m}}

 \def\gz{{\mathfrak z}}

\newcommand{\Id}{{\rm Id}}

\newcommand{\SL}{{\rm SL}}

\newcommand{\vol}{{\it vol}}

\newcommand{\tr}{{\hbox{tr}}}

\newcommand{\Ric}{{\hbox{\it ric}}}
\newcommand{\scal}{{\hbox{\it scal}}}

\newcommand\C{{\mathbb C}}
\newcommand\quater{{\mathbb H}}
\newcommand\R{{\mathbb R}}
\newcommand\Z{{\mathbb Z}}

\def\G{{\Gamma}}
\def\a{{\alpha}}
\def\b{{\beta}}
\def\c{{\gamma}}
\def\d{{\delta}}

\def\m{{\mu}}

\def\s{{\sigma}}
\def\diver{{{\rm div}\,}}
\def\im{{\,{\rm Im}\,}}
\def\ad{{\,{\rm ad}\,}}
\def\Ad{{\,{\rm Ad}\,}}

\begin{document}
\title{Homogeneous para-K\"ahler Einstein manifolds}
\author{D.V. Alekseevsky, C. Medori\and A. Tomassini}
\address{The University of Edinburgh\\
James Clerk Maxwell Building\\
The King's Buildings\\
Mayfield Road\\
Edinburgh\\
EH9 3JZ\\
UK}
\email{D.Aleksee@ed.ac.uk}
\address{Dipartimento di Matematica\\ Universit\`a di Parma\\ Viale G.P. Usberti, 53/A\\ 43100
 Parma\\ Italy}
\email{costantino.medori@unipr.it} \email{adriano.tomassini@unipr.it}
\thanks{This work was partially supported by Leverhulme Trust, EM/9/2005/0069, by the M.I.U.R. Project
``Geometric Properties of Real and Complex Manifolds'' and by G.N.S.A.G.A. of I.N.d.A.M.}
\thanks{The second and the third author would like to thank the School of Mathematics
of the University of Edinburgh for its kind hospitality}
\maketitle

\begin{center}
Dedicated to E.B.Vinberg on the occasion of his 70th birthday
\end{center}
\begin{abstract}
A para-K\"ahler manifold can be defined as a pseudo-Riemannian
manifold $(M,g)$ with a parallel skew-symmetric para-complex
structures $K$, i.e. a parallel field of skew-symmetric
endomorphisms with $ K^2 = \mathrm{Id} $ or, equivalently, as a
symplectic manifold $(M,\omega)$ with a bi-Lagrangian structure
$L^\pm$, i.e. two complementary integrable
Lagrangian distributions. \\
A homogeneous manifold $M = G/H$ of a semisimple Lie group $G$
admits an invariant para-K\"ahler structure $(g,K)$ if and only
if it is a covering of the adjoint orbit $\mathrm{Ad}_Gh$ of a
semisimple element $h.$ We give a description of all invariant
para-K\"ahler structures $(g,K)$ on a such homogeneous manifold.
Using a para-complex analogue of basic formulas of K\"ahler
geometry, we prove that any invariant para-complex structure $K$
on $M = G/H$ defines a unique para-K\"ahler Einstein structure
$(g,K)$ with given non-zero scalar curvature. An explicit formula
for the Einstein metric $g$ is given.

A survey of recent results on para-complex geometry is included.
\end{abstract}
\tableofcontents
\section{Introduction.}
An {\bf almost para-complex structure} on a $2n$-dimensional
manifold $M$ is a field $K$ of involutive endomorphisms ($K^2=1$)
with $n$-dimensional eigendistributions $T^{\pm}$ with
eigenvalues $\pm 1$. \\
(More generally, any field $K$ of involutive
endomorphisms is called an {\bf almost para-complex structure in
weak sense}).\\
If the $n$-dimensional eigendistributions $T^{\pm}$ of $K$ are involutive,
 the field $K$ is called a {\bf para-complex structure}. This is equivalent
 to the vanishing of the Nijenhuis tensor $N_K$ of the $K$. \\
 In other words, a para-complex structure on $M$ is the same as a pair
 of complementary $n$-dimensional integrable distributions $ T^{\pm} M
 $.\\
 A decomposition $M = M_+ \times M_-$ of a manifold $M$ into a
 direct product defines on $M$ a para-complex structure $K$ in the weak
 sense with
 eigendistributions $T^+ = TM_+$ and $T^- = TM_-$ tangent to the
 factors. It is a para-complex structure in the factors $M_{\pm}$
 have the same dimension.\\
 Any para-complex structure locally can be obtained by this
 construction. Due to this, an (almost) para-complex structure is also called
 an {\bf (almost) product structure}.\\
A manifold $M$ endowed with a para-complex structure $K$ admits an
atlas of para-holomorphic coordinates (which are functions with
values in the algebra $C = \R + e \R, \, e^2 =1,$ of para-complex
numbers) such that the transition functions are
para-holomorphic.

One can define para-complex analogues of different composed
geometric structures which involve a complex structure (Hermitian,
K\"ahler, nearly K\"ahler, special K\"ahler, hypercomplex, hyper-K\"ahler,
quaternionic, quaternionic K\"ahler structures, $CR$-structure
etc.) changing a complex structure $J$ to a para-complex
structure $K$. Many results of the geometry of such structures
remain valid in the para-complex case. On the other hand, some
para-complex composed structures admit different interpretation
as 3-webs, bi-Lagrangian structure
and so on.\smallskip

The structure of the paper is the following.\newline
We give a survey of known results about para-complex geometry in section 2. Sections 3,4,5
contain an elementary introduction to para-complex and para-K\"ahler geometry.

The bundle $\Lambda^r(T^*M \otimes C)$ of $C$-valued $r$-form is
decomposed into a direct sum of $(p,q)$-forms $\Lambda^{p,q}M$ and
the exterior differential $d$ is represented as a direct sum $d =
\partial + \bar
\partial$. Moreover, a para-complex analogue of the Dolbeault Lemma
holds (see \cite{CMMS1} and subsection \ref{differentialforms}).\\
A {\bf para-K\"ahler structure} on a manifold $M$ is a pair $(g,K)$
where $g$ is a pseudo-Riemannian metric and $K$ is a parallel
skew-symmetric para-complex structure. A pseudo-Riemannian
$2n$-dimensional manifold $(M,g)$ admits a para-K\"ahler structure
$(g,K)$ if and only if its holonomy group is a subgroup of
$GL_n(\R) \subset SO_{n, n} \subset GL_{2n}(\R)$.\\
If $(g,K)$ is a para-K\"ahler structure on $M$, then $\omega = g \circ K$ is a symplectic structure
and the $\pm 1$-eigendistributions $T^{\pm} M$ of $K$ are two integrable $\omega$-Lagrangian distributions.
 Due to this, a para-K\"ahler structure can be identified with a bi-Lagrangian structure $(\omega, T^{\pm} M)$
 where $\omega$ is a symplectic structure and $T^{\pm}M$ are two integrable Lagrangian distributions.
In section 4 we derive some formulas for the curvature and Ricci curvature of a para-K\"ahler structure
$(g,K)$ in terms of para-holomorphic coordinates. In particular, we show that, as in the K\"ahler case, the Ricci
tensor $S$ depends only on the determinant of the metric tensor $g_{\alpha \bar \beta}$ written in terms of
 para-holomorphic coordinates.\\
In section 5, we consider a homogeneous manifold $(M = G/H, K, {\mathrm{vol}})$ with an invariant para-complex
structure
$K$ and an invariant volume form $\mathrm{vol}$. We establish a formula which expresses the pull-back $\pi^* \rho$ to $G$
of the Ricci form $\rho = S \circ K$ of any invariant para-K\"ahler structure $(g,K)$ as the differential of a left-invariant 1-form $\psi$,
called the {\bf Koszul form}.\\
In the last section, we use the important result by Z.\ Hou, S.\ Deng, S.\ Kaneyuki and K.\ Nishiyama (see \cite{HDK}, \cite{HDKN})
stating that a homogeneous manifold $M= G/H$ of a semisimple Lie group $G$ admits an invariant para-K\"ahler structure if and only if it
is a covering of the adjoint orbit $\Ad_G h$ of a semisimple element $h\in \ggg = {\mathrm{Lie}}(G)$. \\
We describe all invariant para-complex structures $K$ on $M= G/H$ in terms of fundamental gradations of the Lie algebra
$\ggg$ with $\ggg_0 = \gh := \mathrm{Lie}(H)$ and we show that they are consistent with any
invariant symplectic structure $\omega$ on $G/H$ such that $(g = \omega \circ K,\, K)$ is an invariant
para-K\"ahler structure. This gives a description of all invariant para-K\"ahler structures on homogeneous
manifolds of a semisimple group $G$. An invariant para-complex structure on $M = G/H$ defines an Anosov flow,
but a theorem by Y. Benoist and F. Labourie shows that this flow can not be push down to any {\em smooth} compact quotient $\Gamma\setminus G/H$.
 We give a complete description of invariant para-K\"ahler-Einstein metrics on homogeneous manifolds of a
 semisimple Lie group and prove the following theorem.
 \begin{thm} Let $M = G/H$ be a homogeneous manifold of a semisimple Lie group $G$ which admits an invariant
 para-K\"ahler structure
 and $K$ be the invariant paracomplex structure on $M$. Then there exists a unique invariant symplectic structure
$\rho $ which is the push-down of the differential $ d \psi$ of
the Koszul 1-form $\psi$ on $G$ such that
 $ g_{\lambda, K} := \lambda^{-1}\rho\circ K$
is an invariant para-K\"ahler Einstein metric with Einstein constant $\lambda \neq 0$ and this construction gives all
invariant para-K\"ahler-Einstein metrics on $M$.
\end{thm}
\section{A survey on para-complex geometry.}
\subsection{Para-complex structures.}

The notion of {\bf almost para-complex structure} (or {\bf almost product
structure}) on a manifold was introduced by P.K. Ra\v sevski\u{\i}
\cite{Rash} and P. Libermann \cite{L1}, \cite{L2}, where the problem
of integrability had been also discussed. The paper \cite{CFG}
contains a survey of further results on para-complex structure with more then 100 references.
 The papers  \cite{CGM},  \cite{EST}  contain survey of para-Hermitian and para-K\"ahler geometries.
 Different  generalizations of   Hermitian geometry and contact geometry are   considered in   \cite{{Kir}}.\\
 Note that a para-complex structure $K$ on a manifold $M$ defines a new Lie algebra structure
 in the space $\mathfrak{X}(M)$ of vector fields given by
 $$ [X,Y]_K := [KX,Y]+ [X,KY] - K[X,Y]$$
 such that the map
 $$ ( \mathfrak{X}(M),[.,.]_K ) \rightarrow (\mathfrak{X}(M),[.,.]), \, \, X \mapsto KX $$
 is a homomorphism of Lie algebras.\\
 Moreover, $K$ defines a new differential $d_K$ of the algebra $\Lambda(M)$ of differential forms
 which is a derivation of $\Lambda(M)$ of degree one with $d_K^2 =0$. It is given by
 $$d_K := \{d,K \}:= d \circ \iota_K + \iota_K \circ d $$
 where $\iota_K $ is the derivation associated with $K$ of the supercommutative algebra $\Lambda(M)$ of degree $-1$
 defined by the contraction. (Recall that the superbracket of two derivations is a derivation).\\
 In \cite{LS}, the authors define the notion of para-complex affine immersion $f : M \rightarrow M'$
 with transversal bundle $N$ between para-complex manifolds $(M,K), \, (M',K')$ equipped with torsion free
 para-complex connections $\nabla, \nabla'$ and prove a theorem about  existence and
 uniqueness of such an immersion
 of $(M,K,\nabla)$ into the affine space $M' = \mathbb{R}^{2m + 2n}$ with the standard para-complex structure and flat connection.\\
 In \cite{S1}, the notion of para-$tt^*$-bundle over an (almost) para-complex $2n$-dimensional manifold $(M,K)$  is defined  as  a vector bundle $\pi : E \rightarrow M$ with a  connection $\nabla$ and a $\hbox{\rm End}(E)$-valued 1-form
 $S$ such that 1-parametric family of connections
 $$ \nabla^t_X = \nabla_X + \cosh (t) S_X + \sinh (t) S_{KX}, \,\, X \in TM$$
 is flat. It is a para-complex version of $tt^*$-bundle over a complex manifold defined in the context of quantum field theory in \cite{CV} , \cite{D}
 and studied from differential-geometric point of view in \cite{CS}.
 In \cite{S}, \cite{S1}, \cite{S2}, the author studies properties of para-$tt^*$-connection $\nabla^t$ on the tangent
 bundle $E = TM$ of an (almost) para-complex manifold $M$. In particular, he shows that nearly para-K\"ahler and special
 para-complex structures on $M$ provide para-$tt^*$-connections. It is proved also that a para-$tt^*$-connection
 which preserves a metric or a symplectic form determines a para-pluriharmonic map
 $f : M \rightarrow N$ where $N = Sp_{2n}(\mathbb{R})/U^n(C^n)$ or, respectively, $N = SO_{n,n}/U^n(C^n)$
 with invariant pseudo-Riemannian metric. Here $U^n(C^n)$ stands for para-complex analogue of the unitary group.

\subsubsection*{Generalized para-complex structures}
 Let $\mathcal{T}(M) : = TM \oplus T^*M$ be the generalized tangent bundle of a manifold $M$ i.e. the direct sum of the tangent and
 cotangent bundles equipped with the natural metric $g$ of signature $(n,n)$,
$$ g (X, \xi), (X', \xi') := \frac12( \xi(X') + \xi'(X)).
 $$
The { \bf Courant bracket} in the space $\Gamma(\mathcal{T}(M))$ of sections is defined by
$$
[(X,\xi), (X', \xi')] = \left([X,X'], \mathcal{L}_{X}\xi' - \mathcal{L}_{X'}\xi -
\frac12d( \xi'(X) - \xi(X'))\right)
$$
 where $\mathcal{L}_X$ is the Lie derivative in the direction of a vector field $X$.
 A maximally $g$-isotropic subbundle $ D \subset \mathcal{T}(M) $ is called a {\bf Dirac structure} if its
 space of sections $\Gamma(D)$ is closed under the Courant bracket.\\
 Changing in the definition of the para-complex structure the tangent bundle $TM$ to the
 generalized tangent bundle $\mathcal{T}(M)$ and the Nijenhuis bracket to the Courant bracket,
 A. Wade \cite{W} and I. Vaisman \cite{V} define the notion of a
 generalized para-complex structure which unifies the notion of symplectic structure, Poisson structure and
 para-complex structure and similar to the Hitchin's definition of a generalized complex structure (see e.g \cite{G}, \cite{Hi}).\\
 A { \bf generalized para-complex structure} is a field $K$ of involutive skew-symmetric endomorphisms of the bundle $\mathcal{T}(M)$
 whose $\pm 1$-eigendistributions $T^\pm$ are closed under the Courant bracket.\\
 In other words, it is a decomposition
 $\mathcal{T}(M) = T^+ \oplus T^-$ of the generalized tangent bundle into a direct sum of two Dirac structures
 $T^\pm$.\\
 Generalized para-complex structures naturally appear in the context of mirror symmetry: a semi-flat generalized complex
 structure on a $n$-torus bundle with sections over an $n$-dimensional manifold $M$ gives rise to a
 generalized para-complex structure on $M$, see \cite{Ben-B}.
 I. Vaisman \cite{V} extends the reduction theorem of Marsden-Weinstein type to generalized
 complex and para-complex structures and gives the characterization of the submanifolds that inherit an induced structure
 via the corresponding classical tensor fields.
\subsection{Para-hypercomplex (complex product) structures.}
 An {\bf (almost) para-hypercomplex} (or an {\bf almost complex product structure}) on a $2n$-dimensional manifold $M$
 is a pair $(J,K)$ of an anticommuting (almost) complex structure $J$ and an (almost) para-complex structure
 $K$. The product $I = JK$ is another (almost) para-complex structure on $M$. If the structure $J,K$ are integrable,
 then $I =JK$ is also an (integrable) para-complex structure and the pair $(J,K)$ or triple $(I,J,K)$ is called
 a {\bf para-hypercomplex structure}. An (almost) para-hypercomplex structure is similar to an (almost)
 hypercomplex structure which is defined as a pair of anticommuting (almost) complex structures. Like for almost
 hypercomplex structure, there exists a canonical connection $\nabla$, called the {\bf Obata connection},
 which preserves a given almost para-hypercomplex structure (i.e. such that $\nabla J = \nabla K = \nabla I =0$).
 The torsion of this connection vanishes if and only if $N_J = N_K= N_I= 0 $ that is $(J,K)$
 is a para-hypercomplex structure.\\
 At any point
 $x \in M$ the endomorphisms $I,J,K$ define a standard basis of the Lie subalgebra
 $\mathfrak{sl}_2(\mathbb{R}) \subset \mathrm{End}(T_xM)$.
 The conjugation by a (constant)
 matrix $A \in SL_2(\mathbb{R})$ allows to associate with an (almost) para-hypercomplex structure $(J,K)$
 a 3-parametric family of (almost) para-hypercomplex structures which have the same
 Obata connection.\\
 Let $TM = T^+ \oplus T^-$ be the eigenspace decomposition for the almost para-complex structure $K$. Then
 the almost complex structure $J$ defines the isomorphism $J : T^+ \rightarrow T^-$ and we can identify
 the tangent bundle $TM$ with a tensor product $TM = \mathbb{R}^2 \otimes E$ such that the endomorphisms
 $J,K,I$ acts on the first factor $\mathbb{R}^2$ in the standard way :
 \begin{equation} \label{IJKmatrices}
 J= \left(
 \begin{array}{cc}
 0 & -1 \\
 1 & 0 \\
 \end{array}
 \right),\,\,
K= \left(
 \begin{array}{cc}
 1 & 0 \\
 0 & -1 \\
 \end{array}
 \right),\,\,
 I= JK = \left(
  \begin{array}{cc}
  0 & 1 \\
  1& 0 \\
  \end{array}
 \right).
 \end{equation}

 Any basis of $E_x$ defines a basis of the tangent space $T_xM$ and the set of such (adapted) bases
 form a $GL_n(\mathbb{R})$-structure (that is a principal $GL_n(\mathbb{R})$-subbundle of the frame bundle of $M$).
 So one can identify an almost para-hypercomplex structure with a $GL_n(\mathbb{R})$-structure
 and a para-hypercomplex structure with a 1-integrable $GL_n(\mathbb{R})$-structure. This means that
 $(J,K)$ is a para-hypercomplex structure. The basic facts of the geometry of para-hypercomplex manifolds are
 described in \cite{An1}, where also some examples are considered.\newline
 Invariant para-hypercomplex structures on
 Lie groups are investigated in \cite{An} - \cite{AnS}. Algebraically, the construction of left-invariant
 para-hypercomplex structures on a Lie group $G$ reduces to the decomposition of its Lie algebra $\mathfrak{g}$
 into a direct sum of subalgebras $\mathfrak{g}^+, \, \mathfrak{g}^-$ together with the construction of a complex structure
 $J$ which interchanges $\mathfrak{g}^+, \mathfrak{g}^-$. It is proved in \cite{AnS} that the Lie algebras $\mathfrak{g}^{\pm}$
 carry structure of left-symmetric algebra. Applications to construction of hypercomplex and hypersymplectic
 (or para-hyperK\"ahler) structures are considered there.\smallskip

   Connection between a para-hypercomplex structure $(J,K)$ on a $2n$-dimensional manifold
   $M$  and  a {\bf 3-web}  and special $G$-structures are  studied  in \cite{MN}.
   Recall  that  a $3$-web on  a $2n$-dimensional manifold $M$ is a triple $(V_1,V_2,V_3) $
   of mutually complementary $n$-dimensional integrable distributions. A para-hypercomplex structure $(J,K)$  on
   $M$ defines a $3$-web $T^+,T^-,S^+$, where $ T^\pm, S^\pm$ are
  the eigendistributions  of the para-complex structures $K$ and $I = JK$.  Conversely, let $(V_1,V_2,V_3) $ be a 3-web.
   Then the decomposition $TM = V_1 + V_2$
 defines a para-complex structure $K$ and the distribution $V_3$ is the graph of a canonically defined isomorphism
 $f : V_1 \rightarrow V_2$ that is $V_3 = (1 + f)V_1$. The $n$-dimensional distribution $V_4 : = (1 - f^{-1})V_2$
 gives a direct sum decomposition $TM = V_3 + V_4 $ which defines another almost para-complex structure $I$ which anticommute with
 $K$. Hence, $(J = IK, K)$ is almost hypercomplex structure. It is integrable if and only if
 the distribution $V_4 = (1 - f^{-1})V_2$ associated with the 3-web $(V_1,V_2,V_3)$ is integrable.
 So, any 3-web defines an almost para-hypercomplex structure which is para-hypercomplex structure if
 the distribution $V_4$ is integrable. \\
  The monograph \cite{AkS} contains a detailed exposition of the
  theory of three-webs, which was    started  by Bol, Chern and Blaschke and continued by M.A.Akivis and his school.
  Relations with  the theories of $G$-structures, in particular Grassmann structures, symmetric spaces,
  algebraic geometry, nomography, quasigroups, non-associative algebras and  differential equations of
  hydrodynamic type are discussed.
\subsection{Para-quaternionic structures.}

 An {\bf almost para-quaternionic structure} on a $2n$-dimensional manifold $M$ is defined
 by a 3-dimensional subbundle $Q$ of the bundle $\hbox{\rm End}(TM)$ of endomorphisms, which is locally generated
 by an almost para-hypercomplex structure $(I,J,K)$, i.e. $Q_x = \mathbb{R}I_x + \mathbb{R}J_x + \mathbb{R}K_x$
 where $x \in U$ and $U \subset M$ is a domain where $I,J,K$ are defined.
 If $Q$ is invariant under a torsion free connection $\nabla$ (called a {\bf para-quaternionic connection})
 then $Q$ is called a {\bf para-quaternionic structure.} The normalizer of the Lie algebra
 $\mathfrak{sp_1(\mathbb{R})} = \Span(I_x,J_x,K_x)$ in $GL(T_xM)$ is isomorphic to
 $Sp_1(\mathbb{R}) \cdot GL_n(\mathbb{R})$. So an almost para-quaternionic structure can be considered as
 a $Sp_1(\mathbb{R}) \cdot GL_n(\mathbb{R})$-structure and a para-quaternionic structure corresponds to the case when this
 $G$-structure is 1-flat. Any para-hypercomplex structure $(I,J,K)$ defines a para-quaternionic structure
 $Q = \Span(I,J,K)$, since the Obata connection $\nabla$ is a torsion free connection which preserves
 $Q$. The converse claim is not true even locally. A para-quaternionic structure is generated by a
 para-hypercomplex structure if and only if it admits a para-quaternionic connection with holonomy group
 $\Hol \subset GL_n(\mathbb{R})$.\\
 Let $TM = H \otimes E$ be an {\bf almost Grassmann structure} of type $(2,n)$, that is a decomposition of the tangent bundle of a
 manifold $M$ into a tensor product of a 2-dimensional vector bundle $H$ and an $n$-dimensional bundle $E$.
 A non-degenerate 2-form $\omega^H$ in the bundle $H$ defines an almost para-quaternionic structure
 $Q$ in $M$ as follows. For any symplectic basis $(h_-,h_+)$ of a fibre $H_x$ we define
 $ I_x = I^H \otimes 1,\, J_x = J^H \otimes 1,\, K_x = K^H \otimes 1$
 where $I^H, \, J^H, \, K^H$ are endomorphisms of $H_x$ which in the bases $h_-, h_+$ are represented by the
 matrices (\ref{IJKmatrices}). Then, for $x \in M$, $Q_x$ is spanned by $I_x,J_x,K_x$.
\subsection{Almost para-Hermitian and para-Hermitian structures.}
 Like in the complex case, combining an (almost) para-complex structure $K$
 with a "compatible" pseudo-Riemannian metric $g$ we get
 an interesting class of geometric structures. The natural compatibility condition
 is the Hermitian condition which means that that the endomorphism $K$ is skew-symmetric with respect to $g$.
 A pair $(g,K)$ which consists of a pseudo-Riemannian metric $g$ and a skew-symmetric
 (almost) para-complex structure $K$ is called an {\bf (almost) para-Hermitian structure}.
 An almost para-Hermitian structure on a $2n$-dimensional manifold can be identified with
 a $GL_n(\mathbb{R})$-structure.
 The para-Hermitian metric $g$ has necessary the neutral signature $(n,n)$.
 A para-Hermitian manifold $(M,g,K)$ admits a unique connection
 (called the {\bf para-Bismut connection}) which preserves $g,K$ and has a skew-symmetric torsion.\\
 The {\bf K\"ahler form} $\omega : = g \circ K$ of an almost para-Hermitian manifold is not necessary closed. \\
  A special class of para-Hermitian manifolds and their submanifolds  are studied in \cite{KK}. In  the paper \cite{AK}, the authors  classify  autodual 4-dimensional
   para-Hermitian  manifolds  with constant  scalar curvature and parallel Lee form,
    which satisfy  some  extra conditions.\\
 In \cite{GM1}, the authors describe a decomposition of the space of $(0,3)$-tensors
 with the symmetry of covariant derivative $\nabla^g \omega$ (where $\nabla^g$ is the
 Levi-Civita connection of $g$) into irreducible subspaces with respect to the natural action of the
 structure group $GL_n(\mathbb{R})$. It gives an important classification of possible types
 of almost para-Hermitian structures which is an analogue of Gray-Hervella classification of
 Hermitian structures. Such special classes of almost para-Hermitian manifolds are {\bf almost
 para-K\"ahler manifolds} (the K\"ahler form $\omega$ is closed), {\bf nearly para-K\"ahler manifolds}
 ($\nabla^g\omega$ is a 3-form) and {\bf para-K\"ahler manifolds} ($\nabla^g \omega=0$). \smallskip\par
Almost para-Hermitian and almost para-K\"ahler structures naturally arise on the cotangent bundle
 $T^*M$ of a pseudo-Riemannian manifold $(M,g)$. The Levi-Civita connection defines a decomposition
 $T_\xi(T^*M) = T_\xi^{vert}(T^*M) + H_\xi$ of the tangent bundle into vertical and horizontal subbundles. This gives
 an almost para-complex structure $K$ on $T^*M$. The natural identification $T^{vert}_{\xi}M = T^*_xM = T_xM = H_{\xi},$ where
 $\xi \in T^*_xM$ and
 allows to define  also a compatible metric on $T^*M$  that is  almost para-Hermitian  structure, which is studied in  \cite{OPM}.
  Also  the  above  identification   defines an almost complex structure
  $J$ which anticommute with $K$. It  allows to define  an almost para-hyperHermitian structure on $TM$ (i.e. a pseudo-Riemannian metric
  together with two skew-symmetric  anticommuting almost para-complex structures ), studied in \cite{IV}.
 In \cite{BCGHV}, the authors  consider almost para-Hermitian manifolds with pointwise constant para-holomorphic
sectional curvature, which is defined as in the Hermitian case. They characterize these manifolds in terms of the curvature
tensor and prove that Schur lemma is not valid, in general, for an almost para-Hermitian manifold.\\
 A $(1,2)$-tensor field $T$ on an almost para-Hermitian manifold $(M,g,K)$ is called a { \bf homogeneous structure}
 if the connection $\nabla := \nabla^g - T$ preserves the tensors $g,K, T$ and the curvature tensor $R$ of the metric
 $g$. In \cite{GO}, the authors characterize reductive homogeneous manifolds with invariant
 almost para-Hermitian structure in terms of homogeneous structure $T$ and give a classification of possible types of
 such tensors $T$.\\
 Left-invariant para-Hermitian structures on semidirect and twisted products of Lie groups had been constructed and studied
 in the papers   \cite{O1}, \cite{O2}, \cite{O3}, \cite{O}.\\
 $CR$-submanifolds of almost para-Hermitian manifolds are studied in \cite{EFT}.
 Four dimensional compact almost para-Hermitian manifolds are considered  in \cite{Ma}.
 The author decomposes these manifolds into three families and establishes some relations between
 Euler characteristic and Hirzebruch indices of such manifolds.\\
 Harmonic maps between almost para-Hermitian manifolds are considered in \cite{BB}.
 In particular, the authors show that a map $f : M \rightarrow N$ between Riemannian manifolds is
 totally geodesic if and only if the induced map $df : TM \rightarrow TN$ of the tangent bundles
 equipped with the canonical almost para-Hermitian structure is para-holomorphic.\\
 In \cite{Kon} it is proved that the symplectic reduction of an almost para-K\"ahler manifolds
 $(M,g,K)$ under the action of a symmetry group $G$ which admits a momentum map
 $\mu : M \rightarrow \mathfrak{g}^*$ provides an almost para-K\"ahler structure on the
 reduced symplectic manifold $\mu^{-1}(0)/G$. \\
 An almost para-K\"ahler manifold $(M,g, K)$ can be described in terms of symplectic geometry as a
 symplectic manifold $(M, \omega = g \circ K)$ with a bi-Lagrangian
 splitting $TM = T^+ \oplus T^-$, i.e. a decomposition of the tangent bundle into a direct sum of two
 Lagrangian (in general non-integrable) distributions $T^\pm$.
 An almost para-K\"ahler manifold has a canonical symplectic connection $\nabla$ which preserves the
 distributions $T^\pm$, defined by
$$
\begin{array}{l}
\nabla_{X^{\pm}}Y^{\pm} = \frac12 \omega^{-1}(\mathcal{L}_{X^{\pm}} (\omega\circ Y^{\pm}))\,,\\[5pt]
\nabla_{X^+}Y^-= {\mathrm pr}_{T^-}[X^+,Y^-]\,,\quad \nabla_{X^-}Y^+ = {\mathrm pr}_{T^+}[X^-,Y^+]
\end{array}
$$
 for vector fields $X^{\pm},Y^{\pm} \in \Gamma(T^{\pm})$. Note that the torsion $T$ of this connection
satisfies the conditions $T(T^+, T^-)=0$.
 If the distribution $T^+$ is integrable, then $T(T^+,T^+)=0$ and the curvature $R$ satisfies the
 condition $R(T^+,T^+)=0$. So the connection $\nabla$ defines a flat torsion free connection $\nabla^L$
(which depends only on $(\omega,T^+)$) on any leaf $L$ of the integrable distribution $T^+$. So
 the leaf $L$ has the canonical (flat) affine structure.
Indeed, if $X^+ \in \Gamma(T^+)$ is a symplectic vector field, i.e. $\mathcal{L}_{X^+} \omega =0$, then
 $$\nabla^L_{X^+}Y^+= \nabla_{X^+}Y^+ = \frac{1}{2}[X^+, Y^+] , \quad \forall\, Y^+ \in \Gamma(T^+).$$
Since symplectic fields tangent to $T^+$ span the space $\Gamma(T^+)$ of vector fields tangent to $T^+$,
$\nabla^L$ is a well defined flat connection on $L$. Moreover, one can easily check that for symplectic
commuting vector fields $X^+,Y^+ \in \Gamma(T^+)$ and any $Z \in \Gamma(T^-)$, the following holds
\begin{eqnarray*}
R(X^+,Y^+)Z &=& \nabla_{X^+}[Y^+,Z]_{{T^-}}- \nabla_{Y^+}[X^+,Z]_{T^-} \\[2pt]
&=& [X^+,[Y^+,Z] _{T^-}] _{T^-}- [Y^+,[X^+,Z]_{T^-}]_{T^-} \\
&=&  [X^+,[Y^+,Z]] _{T^-}- [Y^+,[X^+,Z]]_{T^-} =0\,,
\end{eqnarray*}
which shows that $R(T^+,T^+)=0$. Here $X_{T^-}$ is the projection of $X \in TM$ onto $T^-$.

 Let $f_1, \ldots, f_n$ be independent functions which are constant on leaves of $T^+$.
 Then the leaves of $T^+$ are level sets $f_1=c_1, \ldots, f_n= c_n$ and the Hamiltonian vector fields
 $X_i:= \omega^{-1}\circ df_i$ commute and form a basis of tangent parallel fields along any leaf $L$.
>From the formula for the Poisson bracket it follows
 $$
 \{f_i,f_j \} = \omega^{-1}(df_i, df_j) =0 = df_i(X_j) = X_j \cdot f_i\,.
 $$
 By a classical theorem of A. Weinstein \cite{Wei} any Lagrangian foliation $\mathcal{T}^+$
 is locally equivalent to the cotangent bundle fibration $T^*N \rightarrow N$. More precisely,
 let $N$ be a Lagrangian submanifold of the symplectic manifold $(M,\omega)$ transversal to the leaves
 of an integrable Lagrangian distribution $\mathcal{T}^+$. Then there is an open neighborhood
 $M'$ of $N$ in $M$ such that $(M', \omega|_{M'}, T^+|_{M'})$ is equivalent to a neighborhood $V$
 of the zero section $Z_N \subset T^*N$ equipped with the standard symplectic structure $\omega_{st}$
 of $T^*N$ and the Lagrangian foliation induced by $T^*N \rightarrow N$. In \cite{V1} a condition in order that
 $(M,\omega, T^+)$ is globally equivalent to the cotangent bundle $T^*N, \omega_{st}$ is given.\\
 Let $(M = T^*, \omega_{st})$ be the cotangent bundle of a manifold $N$ with the standard symplectic structure
 $\omega_{st}$ and the standard integrable Lagrangian distributions $T^+$ defined by the projection $T^*N \rightarrow N$.
 A transversal to $T^+$ Lagrangian submanifolds are the graph $\Gamma_\xi = (x, \xi_x)$ of closed 1-forms
 $\xi \in \Omega^1(N)$. The horizontal distribution $T^-= H^\nabla \subset TM$ of a torsion free linear connection $\nabla$
 in $N$ is a Lagrangian distribution complementary to $T^+$ \cite{Mo1}. Hence any torsion free connection defines
 an almost para-K\"ahler structure on $(\omega_{st}, T^+, T^- = H^{\nabla} )$. Note that the distribution $T^+$ is integrable and
 the distribution $T^{-} = H ^\nabla$ is integrable if and only if the connection $\nabla$ is flat. Such structure is called
 a {\bf half integrable almost para-K\"ahler structure}. An application of such structures to construction of Lax pairs in
 Lagrangian dynamics is given in \cite{CI}.

 In \cite{IZ}, the authors define several canonical para-Hermitian connections on an almost para-Hermitian manifold
 and use them to study properties of 4-dimensional para-Hermitian and 6-dimensional nearly
 para-K\"ahler manifolds. In particular, they show that a nearly para-K\"ahler 6-manifold is Einstein and a priori may be
 Ricci flat and that the Nijenhuis tensor $N_K$ is parallel with respect to the canonical connection.
They also prove that the Kodaira-Thurston surface and the Inoue surface admit a hyper-para-complex
structure. The corresponding para-hyperHermitian structures are locally
(but not globally) conformally para-hyperK\"ahler.
\subsection{Para-K\"ahler (bi-Lagrangian or Lagrangian 2-web) structures.}
 A survey of results on geometry of para-K\"ahler manifolds is given in \cite{EST}. Here we review
 mostly results which are not covered in this paper.
Recall that an almost para-Hermitian manifold $(M,g,K)$ is { \bf para-K\"ahler } if
the Levi-Civita connection $\nabla^g$ preserves $K$ or equivalently its holonomy group
$\Hol_x $ at a point $x \in M$ preserves the eigenspaces decomposition $T_xM = T^+_x + T^-_x$. The parallel
eigendistributions
$T^\pm$ of $K$ are $g$-isotropic integrable distributions. Moreover, they are Lagrangian distributions with respect to
the K\"ahler form $\omega = g \circ K$ which is parallel and, hence, closed.
The leaves of these distributions are totally geodesic submanifolds and they are flat with respect to the
 induced connection,  see section 2.4. \\
 A pseudo-Riemannian manifold
$(M,g)$ admits a para-K\"ahler structure $(g,K)$ if the holonomy group $\Hol_x$ at a point $x$ preserves two
complementary $g$-isotropic subspaces $T^\pm_x$. Indeed the associated distributions $T^\pm$ are parallel and
define the para-complex structure $K$ with $K|_{T^\pm } = \pm 1$.\newline
Let $(M,g)$ be a pseudo-Riemannian manifold. In general, if the holonomy group $\Hol_x$ preserves two complementary
 invariant subspaces $V_1,V_2$, then a theorem by L. Berard-Bergery \cite{Kr} shows that there are two
 complementary invariant $g$-isotropic subspaces $T^\pm_x$ which define a para-K\"ahler structure $(g,K)$
 on $M$.\\
 Many local results of K\"ahler geometry remain valid for para-K\"ahler manifolds, see section 4.
 The curvature tensor of a para-K\"ahler manifold belongs to the space
 $ \mathcal{R}(\mathfrak{gl}_n(\mathbb{R}))$
 of $\mathfrak{gl}_n(\mathbb{R})$-valued 2-forms which satisfy the Bianchi identity. This space decomposes into
 a sum of three irreducible $\mathfrak{gl}_n(\mathbb{R})$-invariant subspaces.
 In particular, the curvature tensor $R$ of a para-K\"ahler manifold has the canonical decomposition
 $$ R = c R_1 + R_{\Ric^0} +W $$
 where $R_1$ is the curvature tensor of constant para-holomorphic curvature 1
 (defined as the sectional curvature in the para-holomorphic direction $(X, KX$)), $R_{\Ric^0}$ is the tensor associated with
 the trace free part $\Ric^0$ of the Ricci tensor $\Ric$ of $R$ and $W \in \mathcal{R}(\mathfrak{sl}_n(\mathbb{R}))$
 is the para- analogue  of the Weyl tensor which has zero Ricci part.
 Para-K\"ahler manifolds with constant para-holomorphic curvature (and the curvature tensor proportional to $R_1$)
 are called {\bf para-K\"ahler space forms}. They were defined in \cite{GM} and studied in \cite{SG},\cite{Er}, \cite{EF}.
 Like in the K\"ahler case, there exists  a unique (up to an isometry) simply connected complete para-K\"ahler manifold $M_k$
 of constant
 para-holomorphic curvature $k$. It can be described as the projective space over para-complex numbers. Any complete
 para-K\"ahler manifold of constant para-holomorphic curvature $k$ is a quotient of $M_k$ by a discrete group of isometries acting
 in a properly discontinuous way.\\
 Different classes of submanifolds of a para-K\"ahler manifold are studied in the following papers:\\
 \cite{AB} (para-complex submanifolds), \cite{EF}, \cite{Ro}, \cite{Do} ($CR$-submanifolds),
 \cite{Ri} (anti-holomorphic totally umbilical submanifolds), \cite{IR} (special totally umbilical submanifolds).
\subsubsection{Symmetric and homogeneous para-K\"ahler manifolds.}
 A para-K\"ahler manifold $(M,g,K)$ is called {\bf symmetric} if there exists a { \bf central symmetry} $S_x$ with center at any point
 $x \in M$ that is an involutive isometry which preserves $K$ and has $x$ as isolated fixed point.
 The connected component $G$ of the group generated by central symmetries (called the { \bf group of transvections})
 acts transitively on $M$ and one can identify
 $M$ with the coset space $M = G/H$ where $H$ is the stabilizer of a point $o \in M$. It is known, see \cite{A},
 that a para-K\"ahler (and, more generally, a pseudo-Riemannian) symmetric space with non  degenerate Ricci  tensor  has
 a semisimple group of isometries. All  such symmetric manifolds are known.  In the Ricci flat case, the group generated by transvections is solvable  and the connected holonomy group is nilpotent \cite{A}.\\
  The classification of
 simply connected symmetric para-K\"ahler manifolds reduces to the classification of { \bf 3-graded Lie algebras}
 $$\mathfrak{g} = \mathfrak{g}^{-1} + \mathfrak{g}^0 + \mathfrak{g}^1,\,\, [\mathfrak{g}^{i}, \mathfrak{g}^j] \subset \mathfrak{g}^{i+j}$$
such that $\mathfrak{g}^0$-modules $\mathfrak{g}^{-1}$ and $\mathfrak{g}^1$ are contragredient (dual).\\
Then $\mathfrak{g} = \mathfrak{h}+ \mathfrak{m}= (\mathfrak{g}^0) + (\mathfrak{g}^{-1}+ \mathfrak{g}^1)$
is a symmetric decomposition, the $\ad_{\mathfrak{h}}$-invariant pairing
$\mathfrak{g}^{-1} \times \mathfrak{g}^1 \rightarrow \mathbb{R}$ defines an invariant metric on the associated
homogeneous manifold $M = G/H$ and the $\ad_{\mathfrak{h}}$-invariant subspaces
$\mathfrak{g}^\pm \subset \mathfrak{m} \simeq T_oM$ define eigendistributions of an invariant para-complex structure $K$.
 Symmetric para-K\"ahler spaces of a semisimple Lie group $G$ were studied and classified by S. Kaneyuki \cite{K1}, \cite{K2}, \cite{KKoz}.\\
 $3$-graded Lie algebras are closely related with Jordan pairs and Jordan algebras, see \cite{Be}, and have different applications.
 In \cite{A}, a construction of some class of Ricci flat para-K\"ahler symmetric spaces of nilpotent Lie groups is described
 and the classification of Ricci flat para-K\"ahler symmetric spaces of dimension 4 and 6 is obtained.\\
 Some classes of para-K\"ahler manifolds which are generalizations of symmetric para-K\"ahler manifolds are defined and
 studied in \cite{DDW} and \cite{DDW1}.

 Homogeneous para-K\"ahler manifolds of a semisimple Lie group had been studied by S. Kaneyuki and his collaborators, see
 \cite{HDK}, \cite{HDKN}. In particular, they prove that a homogeneous manifold $M = G/H$ of a semisimple Lie group $G$ admits an
 invariant para-K\"ahler structure if and only if it is a cover of the adjoint orbit of a semisimple element of the Lie algebra
 $\mathfrak{g}$ of $G$. A description of invariant para-K\"ahler structures on homogeneous manifolds of the normal real form of a
 complex semisimple Lie group was given in \cite{Hou}, \cite{HD}. An explicit description of all invariant
 para-K\"ahler structures on homogeneous manifolds  of  a semisimple  Lie  group had been given in \cite{AM}, see section 6.\\
Berezin quantization on para-K\"ahler symmetric spaces of a semisimple Lie group had been studied in
\cite{M}, \cite{MV}, \cite{MV1}. Kostant quantization of a general symplectic manifold with a bi-Lagrangian structure
(that is a para-K\"ahler manifold) is considered in \cite{Hess}.
\subsubsection{Special para-K\"ahler manifolds, $tt^*$-bundles and affine hyperspheres.}
 The notion of {\bf special para-K\"ahler structure} on a manifold $M$ had been defined in the important
 paper \cite{CMMS1}. It is an analogue of the notion of special K\"ahler structure and
 it is defined as a para-K\"ahler structure $(g,K)$ together with a flat torsion free connection
 $\nabla$ which preserves the K\"ahler form $\omega= g \circ K$ and satisfies the condition
 $(\nabla_X K)Y = (\nabla_Y K)X$ for $X,Y \in TM$.

 It was shown in \cite{CMMS1}, \cite{CMMS2} that
 the target space for scalar fields in 4-dimensional Euclidean $N=2$ supersymmetry carries a special
 para-K\"ahler structure similar to the special K\"ahler structure which arises on the target space of
 scalar fields for $N=2$ Lorentzian 4-dimensional supersymmetry. This gives an explanation of a formal
 construction \cite{GGP}, where the authors obtains the Euclidean supergravity action changing in the Lagrangian
 the imaginary unit $i$ by a symbol $e$ with $e^2 =1$. Besides physical applications, \cite{CMMS1}
 contains an exposition of basic results of para-complex and para-K\"ahler geometry in para-holomorphic coordinates.\\
 In \cite{CLS}, the authors construct a canonical immersion of a simply connected special para-K\"ahler manifold
 into the affine space $\mathbb{R}^{n+1}$ as a parabolic affine hypersphere. Any non-degenerate para-holomorphic function
 defines a special para-K\"ahler manifold and the corresponding affine hypersphere is explicitly described.
 It is shown also that any conical special para-K\"ahler manifold is foliated by proper affine hyperspheres of constant
 affine and pseudo-Riemannian mean curvature. Special para-K\"ahler structures are closely related with
 para-$tt^*$-bundles,\cite{S1,S2,S}.
\subsection{Para-hyperK\"ahler (hypersymplectic) structures and para-hyperK\"ahler structures with torsion (PHKT-structures).}

A {\bf para-hyperK\"ahler} ({\bf hypersymplectic}) structure on a $4n$-dimensional manifold $M$
 can be defined as a pseudo-Riemannian metric $g$ together with a $\nabla^g$-parallel
 $g$-skew-symmetric para-hypercomplex structure $(J,K)$. A pseudo-Riemannian metric $g$ admits a
 para-hyperK\"ahler  structure if its holonomy group is a subgroup of the symplectic group
 $Sp_n(\mathbb{R}) \subset SL_{4n}(\mathbb{R})$. A para-hypercomplex structure $(J,K)$ defines
 a para-hyperK\"ahler structure if and only if its (torsion free) Obata connection $\nabla$
 preserves a metric $g$. In other words, the holonomy group of $\nabla$ ( which is in general a subgroup
 of the group $ GL(E) \simeq GL_{2n}(\mathbb{R})$, which acts on the second factor of the canonical decomposition
 $T_xM = \mathbb{R}^2 \otimes E$) must preserve a non-degenerate 2-form $\omega^E$, hence be a subgroup of
 $Sp(E, \omega^E) \simeq Sp_{n}(\mathbb{R})$.
 The metric $g$ of a para-hyperK\"ahler structure has signature $(2n, 2n).$ In this case the $2$-form $\omega^E$
together with the volume $2$-form $\omega^H$ on the trivial bundle $\mathbb{R}^2 \times M \rightarrow M$
defines the metric $g = \omega^H \otimes \omega^E$ in $TM = \mathbb{R}^2 \otimes E$ such that
$(g, J,K)$ is a para-hyperK\"ahler structure.\\
Para-hyperK\"ahler structure $(g,J,K)$ can be also described in symplectic terms as follows.
Note that $\omega_1 := g \circ(JK),\, \omega_2 := g \circ J, \, \omega_3 = g \circ K$ are three
parallel symplectic structures and the associated three fields of endomorphisms
$$
K= \omega_1^{-1} \circ \omega_2, \quad -I= \omega_2^{-1} \circ \omega_3,\quad -J= \omega_3^{-1} \circ \omega_1
$$
define a para-hypercomplex structure $(I =JK,J,K)$. Conversely, three symplectic structures
$\omega_1, \omega_2, \omega_3$ such that the associated three fields of endomorphisms $I,J,K$ form an almost
para-hypercomplex structure $(I,J,K)$ and define a para-hyperK\"ahler structure $(g = \omega_1 \circ I , J,K)$,
see \cite{AnS}. In the hyperK\"ahler case this claim is known as  Hitchin Lemma. Due to this, para-hyperK\"ahler manifolds
are also called {\bf hypersymplectic manifolds}.\newline
 The metric of a hypersymplectic manifold is Ricci-flat and the space of curvature tensors
 can be identified with the space $S^4E$ of homogeneous polynomial of degree four. \\
 In contrast to the hyperK\"ahler case, there are homogeneous and even symmetric
 hypersymplectic manifolds. However, the isometry group of a symmetric hypersymplectic manifold is solvable.
 In \cite{ABCV}, a classification of simply connected symmetric para-hyperK\"ahler
 manifolds with commutative holonomy group is presented. Such manifolds are defined by homogeneous polynomials of degree 4.\newline
 In \cite{FPPW}, the authors construct hypersymplectic structures on some Kodaira manifolds. Many other
 interesting examples of left-invariant hypersymplectic structures of solvable Lie groups
 had been given in \cite{An}, \cite{AnS} , \cite{AD}  and  \cite{ADBO}, where all such structures on 4-dimensional Lie groups
 had been classified.  Under some  additional assumption,  it is given also in
 \cite{BV}.\newline
In \cite{DS}, the authors construct and study hypersymplectic spaces obtained as quotients of the flat hypersymplectic
space $\mathbb{R}^{4n}$ by the action of a compact abelian group. In particular, conditions for smoothness
and non-degeneracy of such quotients are given.

A natural generalization of para-hyperK\"ahler structure is a {\bf para-hyperK\"ahler structure with torsion (PHKT-structure)}
defined and studied in \cite{ITZ}. It is defined as a pseudo-Riemannian metric $g$ together with
a skew-symmetric para-hypercomplex structure $(J,K)$ such that there is a connection $\nabla$ which preserves
$g,J,K$ and has skew-symmetric torsion tensor $T$, i.e. such that $g \circ T$ is a 3-form. The structure is called {\bf strong} if
the 3-form $g \circ T$ is closed. The authors show that locally such a structure is defined by a real function (potential) and
construct examples of strong and non strong PHKT-structures.
\subsection{Para-quaternionic K\"ahler structures and para-quaternionic-K\"ahler
with torsion (PQKT) structures.}
A para-quaternionic structure on a $4n$-dimensional manifold $M^{4n}$ can be defined as
 a pseudo-Riemannian metric $g$ (of neutral signature $(2n,2n)$) with holonomy group in $Sp_1(\mathbb{R}) \cdot Sp_n(\mathbb{R})$.
 This means that the Levi-Civita connection $\nabla^g$ preserves a para-quaternionic structure $Q$.
 This implies that the metric $g$ is Einstein. Moreover, the scalar curvature $\scal$ is zero if and only if the restricted holonomy group is in
 $Sp_n(\mathbb{R})$ and the induced connection in $Q$ is flat. \\
 We will assume that $\scal \neq 0$. Then the holonomy group contains $Sp_1(\mathbb{R})$ and it is irreducible.
 Examples of para-quaternionic K\"ahler manifolds are para-quaternionic K\"ahler symmetric spaces.
 Any para-quaternionic K\"ahler symmetric space is a homogeneous manifold $M = G/H$ of a simple Lie group $G$ of isometries.
 The classification of simply connected para-quaternionic K\"ahler symmetric spaces reduces to the
 description of the $\mathbb{Z}$-gradations
 of the corresponding real simple Lie algebras of the form
 $$ \mathfrak{g} = \mathfrak{g}^{-2} + \mathfrak{g}^{-1}+ \mathfrak{g}^0 + \mathfrak{g}^1 + \mathfrak{g}^2 $$
with $ \dim \mathfrak{g}^{\pm 2}=1$.
 Such gradations can be easily determined, see section 6. The classification of para-K\"ahler symmetric
 spaces based on twistor approach is given in \cite{DJS}. In this paper, the authors define the {\bf twistor spaces}
 $Z(M)$, the {\bf para-3-Sasakian bundle} $S(M)$ and the {\bf positive Swann bundle} $U^+(M)$ of a para-quaternionic
 K\"ahler manifold $(M,g,Q)$ and study induced geometric structures.
 They assume that $M$ is real analytic and consider a (locally defined) holomorphic Grassmann structure of the
 complexification $M^{\mathbb{C}}$ of $M$, that is an isomorphism $TM^{\mathbb{C}} \simeq E \otimes H$ of the holomorphic
 tangent bundle $TM^{\mathbb{C}}$ with the tensor product of two holomorphic vector bundles $E,H$
 of dimension $2n$ and $2$. The {\bf twistor space} is defined as the projectivization $Z(M) := PH$ of the bundle $H$
 which can be identified with the space of (totally geodesic) $\alpha$-surfaces in $M^{\mathbb{C}}$. The authors prove that
 $Z(M)$ carries a natural holomorphic contact form $\theta $ and a real structure and give a characterization
 of the twistor spaces $Z = Z(M)$ as $(2n+1)$-dimensional complex manifolds with a holomorphic contact form and a real structure
 which satisfy some conditions.\\
 This is a specification of a more general construction \cite{BE} of the twistor space of a manifold with a
{\bf holomorphic Grassmann structure} (called also a {\bf complex para-conformal structure}).\\
 The {\bf para-3-Sasakian bundle} associated with $(M,g,Q)$ is the principal $SL_2(\mathbb{R})$-bundle $S(M) \rightarrow M$
 of the standard frames $ s = (I,J,K)$ of the quaternionic bundle $Q \subset \hbox{\rm End}(TM)$. It has a natural metric $g_S$ defined
 by the metric $g$ and the standard bi-invariant metric on $SL_2(\mathbb{R})$. The standard generators of $SL_2(\mathbb{R})$
 define three Killing vector fields $\xi_{\alpha}, \, \alpha =1,2,3$ with the square norm $1,-1,1$ which
 define a {\bf para-3-Sasakian structure}. This means that the curvature operator $R(X, \xi_\alpha) = X \wedge \xi_\alpha$ for $X \in TM$.
 This is equivalent to the condition that the cone metric $g_U = dr^2 + r^2 g$ on the cone $U^+(M) = M \times \mathbb{R}^+$
 is hypersymplectic, i.e. admits a $\nabla^{g_U}$-parallel para-hypercomplex structure $I_U,J_U,K_U$.
 The bundle $U^+(M) \rightarrow M$ is called the {\bf positive Swann bundle}.\smallskip

 Let $G^{\mathbb{C}}$ be the complexification of a real non compact semisimple Lie group $G$  and
 $\mathcal{O}$ an adjoint nilpotent orbit of $G^{\mathbb{C}}$. \\
 In \cite{DJS}, it is proved that the projectivization
 $\mathbb{P}\mathcal{O}$ is the twistor space of a para-quaternionic K\"ahler manifold $M$ if and only if
 $\mathcal{O}$ is the complexification of a nilpotent orbit of $G$ in the real Lie algebra $\mathfrak{g}$. In this case,
 $G$ acts transitively on $M$ preserving the para-quaternionic K\"ahler structure.

 Using the hypersymplectic momentum map of $U^+(M)$, the authors prove that, if a semisimple isometry group $G$ of a
 para-quaternionic K\"ahler manifold $(M,g,Q)$ acts transitively on the twistor space $Z(M)$, then the corresponding orbit
 $\mathcal{O}$ is the orbit of the highest root vector and the manifold $M$ is symmetric. This leads to the classification of
 para-quaternionic K\"ahler symmetric spaces of a semisimple Lie group $G$, hence of all para-quaternionic K\"ahler symmetric
 spaces of non-zero scalar curvature, due to the following result:\\
 Any para-quaternionic K\"ahler or pseudo-quaternionic K\"ahler symmetric space with a
 non-zero scalar curvature $\scal$ is a homogeneous space of a simple Lie group. \\
 This result have  been  proven in \cite{AC1}, where a classification of pseudo-quaternionic K\"ahler symmetric spaces with $\scal \neq 0$ is given in terms of Kac diagrams.
 In \cite{Bl} and \cite{BDM}, the authors define two twistor spaces $Z^+(M)$ and $Z^-(M)$ of a para-quaternionic K\"ahler manifold
 $(M,g,Q)$ of dimension $4n$. They consist of all para-complex structures $K \in Q, \, K^2 = 1$ and, respectively, complex
 structures $J \in Q,\, J^2 =-1$ of the quaternionic bundle $Q$. The authors define a natural almost para-complex, respectively,
 an almost complex structure on $Z^+(M)$ and, respectively, $Z^-(M)$ and prove their integrability in the case $n >1$.
 In \cite{AC}, the geometry of these twistor spaces are studied using the theory of $G$-structures. It is proved that the natural
 horizontal distribution on $Z^{\pm}(M)$ is a para-holomorphic or, respectively, holomorphic contact distribution and that
 the manifolds $Z^{\pm}(M)$ carries two canonical Einstein metrics, one of which is (para)-K\"ahler. A twistor description of
 (automatically minimal) K\"ahler and para-K\"ahler submanifolds in $M$ is given.\\
 A {\bf para-quaternionic K\"ahler manifold with torsion} is a pseudo-Riemannian manifold $(M,g)$ with a skew-symmetric almost
 para-quaternionic structure $Q$ which admits a preserving $Q$ metric connection with skew-symmetric torsion.\\
 Such manifolds are studied in \cite{Z}. In particular, it is proved that this notion is invariant under a conformal transformation of the metric $g$.
\subsection{Para-CR structures and para-quaternionic CR structures (para-3-Sasakian structures).}
A {\bf weak almost para-$CR$ structure of codimension $k$} on a $m +k$-dimensional
manifold $M$ is a pair $(HM,K)$, where $HM \subset TM$ is a rank
$m$ distribution and $ K \in {\rm End}(HM)$ is a field of
endomorphisms such that $K^2=\hbox{\rm id}$ and $K\neq\pm\hbox{\rm
id}$. If $m =2n$  and $\pm 1$-eigendistributions $H^{\pm}$ of $K$ has rank $n$,
the pair $(H,K)$ is called  an almost para-$CR$-structure.
\newline A (weak) almost para-CR structure is said to
be a ({\bf weak}) {\bf para-CR structure}, if it is {\rm  integrable},
that is eigen-distributions $H^{\pm}$ are involutive or, equivalently, the
 the following conditions hold:
\begin{eqnarray*}
&{}&[KX,KY]+[X,Y]\in\Gamma(HM)\,,\label{integrability1}\\[3pt]
&{}& S(X,Y):=[X,Y]+[KX,KY] -
K([X,KY]+[KX,Y])=0\label{integrability2}
\end{eqnarray*}
for all  $X,\,Y\in\Gamma(HM)$.\newline
 In \cite{AMT} and \cite{AMT1}, a description of maximally homogeneous weak para-CR-structures of semisimple  type in
  terms of fundamental gradations of   real semisimple Lie algebras  is  given.\\
  Codimension one para-CR structures  are naturally  arise on generic  hypersurfaces of a para-complex manifold
  $N$, in particular, on hypersurfaces in  a  Cartesian  square $N =X \times X $ of a manifold $X$. \\
  Consider, for example, a second order ODE $\ddot y = F(x,y,\dot y)$. Its general solution
  $y = f(x,y, X,Y)$ depends on two free parameters $X,Y$ (constants of integration) and determines a hypersurface
  $M$ in the space $\R^2 \times \R^2 = \{(x,y,X,Y) \}$ with the  natural para-complex structure $K$ invariant
  under the point  transformations. The  induced para-$CR$ structure on the space $M$ of solutions  plays important role in
  a geometric theory of ODE, developed in \cite{NS}, where a para-analogue of the Fefferman metric on a bundle over
  $M$ is constructed and a notion of  duality of ODEs  is introduced and studied.

In \cite{AKam}, a notion of a {\bf para-quaternionic} (respectively,{ \bf quaternionic}) {\bf CR structure} on
$4n+3$-dimensional manifold $M$
is defined as a triple
$(\omega_1, \omega_2, \omega_3)$ of 1-forms such that associated 2-forms
$$
\rho^1 = d\omega_1 - 2 \epsilon \,\omega_2 \wedge \omega_3\,,\quad \rho^2 = d\omega_1 + 2\omega_3 \wedge \omega_1\,,
\quad\rho^3 = d\omega_3 + 2 \omega_1 \wedge \omega_2\,,
$$
(where $\epsilon = 1$ in the para-quaternionic case and $\epsilon =-1$ in the quaternionic case)
are non-degenerate on the subbundle $H =\cap_{\alpha =1}^3 \hbox{\rm Ker}\, \omega_\alpha $
of codimension three and the endomorphisms
$$
 J_1 = - \epsilon(\rho^3_H )^{-1} \circ \rho^2_H\,,\quad J_2 = (\rho^1_H )^{-1} \circ \rho^3_H\,,\quad
J_3 = (\rho^2_H )^{-1} \circ \rho^1_H\,,
$$
define an almost para-hypercomplex (respectively, almost hypercomplex) structure on $H$.
It is shown that such a structure defines a para-$3$-Sasakian (respectively, pseudo-$3$-Sasakian)
structure on $M$, hypersymplectic (respectively, hyperK\"ahler) structure on the cone $C(M) = M \times \mathbb{R}^+$
and, under some regularity assumptions, a para-quaternionic K\"ahler (respectively, quaternionic K\"ahler)
structure on the orbit space of a naturally defined $3$-dimensional group of transformations on $M$
(locally isomorphic to $Sp_1(\mathbb{R})$ or $Sp_1$). Some homogeneous examples of such structures are indicated and
a simple reduction method for constructing a non-homogeneous example is described.
\section{Para-complex vector spaces.}
\subsection{The algebra of para-complex numbers $C$.}
We recall that the {\bf algebra of para-complex numbers}
is defined as the vector space $C = \R^2$ with the multiplication
$$
(x,y)\cdot (x',y')= (xx'+yy',xy'+yx')\,.
$$
 We set $e=(0,1)$. Then $e^2=1$ and we can write
$$
C = \R + e \R = \{z = x + ey\,,\,\,\vert\,\,x,y\in\R\}\,.
$$
The conjugation of an element $z = x + ey$ is defined by $\bar z := x -ey$
and $ \Re\mathfrak{e}\,z :=x $ and
$\Im\mathfrak{m}\,z=y$ are called the {\bf real part} and the {\bf imaginary part} of the para-complex number $z$, respectively.\newline
We denote by $C^* = \{z = x +ey\,, \,\, \vert\,\,x^2-y^2\neq 0 \}$ the group of invertible elements of $C$.\smallskip

\subsection{Para-complex structures on a real vector space.}
Let $V$ be a $2n$-dimensional real vector space. A {\bf para-complex structure}
on $V$ is an endomorphism $K:V\to V$ such that
\begin{itemize}
\item[i)] $K^2=\Id_V$;
\item[ii)] the eigenspaces $V^+$, $V^-$ of $K$ with eigenvalues $1$, $-1$, respectively, have the same dimension.
\end{itemize}
The pair $(V,K)$ will be called a {\bf para-complex vector space}.\newline
Let $K$ be a para-complex structure on $V$. We define the {\bf para-complexification} of $V$ as $V^C=V\otimes_\R C$ and we extend $K$ to
a $C$-linear endomorphism $K$ of $V^C$. Then by setting
\begin{eqnarray*}
V^{1,0}&=&\{v\in V^C\,\,\vert\,\, Kv=ev\}=\{v+eKv\,\,\vert\,\, v\in V\}\,,\\
V^{0,1}&=&\{v\in V^C\,\,\vert\,\, Kv=-ev\}=\{v-eKv\,\,\vert\,\, v\in V\}\,,
\end{eqnarray*}
we obtain $V^C=V^{1,0}\oplus V^{0,1}$.
\begin{rem}\label{module} {\rm The extension $K$ of the endomorphism $K$ with eigenvalues $\pm 1$ to $V^C$ has
''eigenvalues'' $\pm e$. There is no
contradiction since $V^C$ is a module over $C$, but not a vector space ($C$ is an algebra, but not a field).}
\end{rem}
The {\bf conjugation} of $x+ey\in V^C$ is given by
$\overline{z}=x-ey$, where $x,y\in V$. A para-complex structure $K$ on $V$ induces a para-complex structure $K^*$ on the dual
space $V^*$ of $V$ by
$$
K^*(\alpha)(v)=\alpha(Kv)\,,
$$
for any $\alpha\in V^*$, $v\in V$. Therefore, if $V^{* C}=V^*\otimes_\R C$, then $V^{*C}=V_{1,0}\oplus V_{0,1}$, where
\begin{eqnarray*}
V_{1,0}&=&\{\alpha\in V^{*C}\,\,\vert\,\, K^*\alpha=e\alpha\}=\{\alpha+eK^*\alpha\,\,\vert\,\, \alpha\in V^*\}\,,\\
V_{0,1}&=&\{\alpha\in V^{*C}\,\,\vert\,\, K^*\alpha=-e\alpha\}=\{\alpha-eK^*\alpha\,\,\vert\,\, \alpha\in V^*\}\,.
\end{eqnarray*}
Denote by $\wedge^{p,q}V^{*C}$ the subspace of $\wedge V^{*C}$ spanned by $\alpha\wedge\beta$, with $\alpha\in\wedge^pV_{1,0}$ and
$\beta\in\wedge^qV_{0,1}$. Then
$$
\wedge^rV^{*C} =\bigoplus_{p+q=r}\wedge^{p,q}V^{*C}\,.
$$
If $\{e^1,\ldots ,e^n\}$ is a basis of $V_{1,0}$, then $\{\overline{e}^1,\ldots ,\overline{e}^n\}$ is a basis of $V_{0,1}$ and
$$
\{e^{i_1}\wedge\cdots\wedge e^{i_p}\wedge\overline{e}^{j_1}\wedge\cdots\wedge\overline{e}^{j_q},\,\,1\leq i_1<\cdots <i_p\leq n\,,\,\,
1\leq j_1<\cdots <j_q\leq n\}
$$
is a basis of $\wedge^{p,q}V^{*C}$.
\subsection{Para-Hermitian forms.}
\begin{definition} A {\bf para-Hermitian form} on $V^C$ is a map $h:V^C\times V^C\to C$ such that:
\begin{itemize}
\item[i)] $h$ is $C$-linear in the first entry and $C$-antilinear in the second entry;
\item[ii)] $h(W,Z)=\overline{h(Z,W)}$.
\end{itemize}
\end{definition}
\begin{definition}
A {\bf para-Hermitian symmetric form} on $V^C$ is a symmetric $C$-bilinear form $h:V^C\times V^C\to C$ such that
\begin{eqnarray*}
h(V^{1,0},V^{1,0})&=&h(V^{0,1},V^{0,1}) =0\,,\\
h(\overline{Z},\overline{W}) &=&\overline{h(Z,W)}
\end{eqnarray*}
for any $Z,W\in V^C$.\newline
It is called {\bf non-degenerate} if it has trivial kernel:
$$
\ker (h) = \left\{ Z\in V^C\,\,: \,\, h(Z,V^C)=0 \right\}=0
$$
\end{definition}
If $h(Z,W)$ is a para-Hermitian symmetric form, then $\hat{h}(Z,W)=h(Z,\overline{W})$
is a para-Hermitian form.
\begin{lemma} There exists a natural $1-1$ correspondence between pseudo-Euclidean metric $g$ on a
vector space $V$ such that
$$
g(KX,KY)=-g(X,Y)\,,\quad X,Y\in V
$$
and non-degenerate para-Hermitian symmetric forms $h=g^C$ in $V^C$, where $g^C$ is the natural extension of $g$ to
$C$-bilinear symmetric form. Moreover, the natural $C$-extension $\omega^C$ of the two form $\omega =g\circ K$ coincides with the $(1,1)$-form
$g^C\circ K$.
\end{lemma}
\section{Para-complex manifolds.}
We recall some basic definitions of para-complex geometry.
\begin{definition} An {\bf almost para-complex structure} on a $2n$-dimensional
manifold $M$ is a field $K\in \Gamma (\hbox{\rm End}(TM))$ of endomorphisms such that
\begin{itemize}
\item[i)]$K^2=\Id_{TM}$,
\item[ii)] the two eigendistributions $T^\pm M:=\ker(\Id\mp K)$ have the same rank\,.
\end{itemize}
A {\bf almost para-complex structure} $K$ is said to be {\bf integrable} if the distributions $T^\pm M$ are involutive.
In such a case $K$ is called a {\bf para-complex structure}. A manifold $M$ endowed with an (almost) para-complex structure $K$ is
called an {\bf (almost) para-complex manifold}.\newline
A map $f:(M,K)\to(M',K')$ between two (almost) para-complex manifolds is said to be {\bf para-holomorphic} if
\begin{equation}\label{paraholomorphicmap}
df\circ K=K'\circ df\,.
\end{equation}
\end{definition}
The {\bf Nijenhuis tensor} $N_K$ of an almost para-complex structure $K$ is defined by
$$
N_K(X,Y)=[X,Y]+[KX,KY]-K[KX,Y]-K[X,KY]\,
$$
for any vector fields $X$, $Y$ on $M$. As in the complex case, a para-complex structure $K$ is integrable if and only
if $N_K=0$ (see e.g. \cite{CMMS1}).\newline
A basic example of a para-complex manifold is given by
$$
C^n:=\{(z^1,\ldots ,z^n)\,\,\vert\,\,z^\alpha\in C\,,i =1,\ldots ,n\}\,,
$$
where the para-complex structure is provided by the multiplication by $e$.
The Frobenius theorem implies (see e.g. \cite{CMMS2}) the existence of local coordinates
$(z^\alpha_+,z^\alpha_-)\,\,,\alpha =1,\ldots , n$ on a para-complex manifold $(M,K)$, such that
$$
T^+M=\hbox{\rm span}\left\{\frac{\partial}{\partial z^\alpha_+}\,,\alpha =1,\ldots , n\right\}\,,\quad
T^-M=\hbox{\rm span}\left\{\frac{\partial}{\partial z^\alpha_-}\,,\alpha =1,\ldots , n\right\}\,.
$$
Such (real) coordinates are called {\bf adapted coordinates} for the para-complex structure $K$.\newline
The cotangent bundle $T^*M$ splits as
$T^*M = T^{*}_+M\oplus T^{*}_-M$, where $T^{*}_{\pm}M$ are the $\pm 1$-eigendistributions of $K^*$. Therefore,
$$
\wedge^rT^*M = \bigoplus_{p+q=1}^r\wedge^{p,q}_{+\,-}T^*M\,,
$$
where $\wedge^{p,q}_{+\,-}T^*M =\wedge^p(T^*_+ M)\otimes\wedge^q(T^*_- M)$.
The sections of $\wedge^{p,q}_{+\,-}T^*M$ are called {\bf $(p+,q-)$-forms} on the para-complex
manifold $(M,K)$. We will denote the space of sections of the bundle $\wedge^{p,q}_{+\,-}T^*M$ by the same symbol.\newline
We set
\begin{eqnarray*}
\partial_+ &=&\hbox{\rm pr}_{\wedge^{p+1,q}_{+\,-}(M)}\circ d : \wedge^{p,q}_{+\,-}(M)\to \wedge^{p+1,q}_{+\,-}(M)\,, \\
\partial_- &=&\hbox{\rm pr}_{\wedge^{p,q+1}_{+\,-}(M)}\circ d : \wedge^{p,q}_{+\,-}(M)\to \wedge^{p,q+1}_{+\,-}(M)\,.
\end{eqnarray*}
Then the exterior differential $d$ can be decomposed as $d =\partial_+ + \partial_-$ and, since $d^2=0$, we have
$$
\partial^2_+=\partial^2_- =0\,,\quad \partial_+\partial_-+\partial_-\partial_+=0\,.
$$
\subsection{Para-holomorphic coordinates.}
Let $(M,K)$ be a $2n$-dimensional para-complex manifold. Like in the complex case we can define on $M$ an atlas of para-holomorphic
local charts $(U_a,\varphi_a)$ where $\varphi_a : M\supset U_a\to C^n$ and such that the transition functions
$\varphi_a\circ\varphi^{-1}_b$ are para-holomorphic functions in the sense of \eqref{paraholomorphicmap}. We
associate with any adapted coordinate system $(z^\alpha_+,z^\alpha_-)$ a para-holomorphic coordinate
system $z^\alpha$ by
\begin{equation}\label{coordinateparacomplesse}
z^\alpha=\frac{z^\alpha_++z^\alpha_-}{2} + e\,\frac{z^\alpha_+-z^\alpha_-}{2}\,,\quad \alpha=1,\ldots ,n\,.
\end{equation}
One can easily check (see \cite{CMMS1}) that $z^\alpha$ are para-holomorphic functions in the sense
of \eqref{paraholomorphicmap} and that the transition functions between two para-holomorphic coordinate systems are
para-holomorphic.
We stress that the real part $x^\alpha$ and the imaginary part $y^\alpha$ of the
functions $z^\alpha$ given by
$$
x^\alpha = \frac{1}{2}(z^\alpha +\overline{z}^\alpha)=\frac{1}{2}(z^\alpha_++z^\alpha_-)\,,\,\,\,
y^\alpha = \frac{1}{2e}(z^\alpha -\overline{z}^\alpha)=\frac{1}{2}(z^\alpha_+-z^\alpha_-)
$$
are not necessarily real analytic.
\subsection{Para-complex differential forms.}\label{differentialforms}
Let $(M,K$) be a para-complex manifold. We define {\bf para-complex tangent bundle}
as the $\R$-tensor product $T^CM=TM\otimes C$ and we extend the endomorphism $K$
to a $C$-linear endomorphism of $T^CM$. For any $p\in M$,
we have the following decomposition of $T_p^CM$:
\begin{equation}
\label{1001}
T_p^CM=T^{1,0}_pM\oplus T^{0,1}_pM\,,
\end{equation}
where
\begin{eqnarray*}
T^{1,0}_pM&=&\{Z\in T_p^CM\,\,\vert\,\, KZ=eZ\}=\{X+eKX\,\,\vert\,\, X\in T_pM\}\,,\\
T^{0,1}_pM&=&\{Z\in T_p^CM\,\,\vert\,\, KZ=-eZ\}=\{X-eKX\,\,\vert\,\, X\in T_pM\}\,
\end{eqnarray*}
are the ''eigenspaces'' of $K$ with ''eigenvalues'' $\pm e$ (see remark \ref{module}).\newline
We define the {\bf conjugation} of an element $Z=X+eY\in T_pM^C$ by $\overline{Z}=X-eY$. Then
$T_p^{0,1}M=\overline{T}_p^{\,1,0}M$. The para-complex vectors
$$
\frac{\partial}{\partial z^\alpha}=\frac{1}{2}\left(\frac{\partial}{\partial x^\alpha}+e\frac{\partial}{\partial y^\alpha}\right)\,,\quad
\frac{\partial}{\partial \overline{z}^\alpha}=\frac{1}{2}\left(\frac{\partial}{\partial x^\alpha}-e\frac{\partial}{\partial y^\alpha}\right)\,
$$
form a basis of the spaces $T^{1,0}_pM$ and $T^{0,1}_pM$, respectively. A vector
$Z\in T^{1,0}_pM$, respectively, $ \bar Z \in T^{0,1}_pM$, has uniquely defined
coordinates with respect to $\frac{\partial}{\partial z^\alpha}$, respectively,
 $\frac{\partial}{\partial \overline{z}^\alpha}$.\newline
The para-complex structure $K$ acts on the dual space $(T^C)^*M$ by
$$
K^*\alpha(X)=\alpha(KX)\,.
$$
We have a decomposition
$$
(T^C)^*M =\wedge^{1,0}(M) \oplus \wedge^{0,1}(M)\,,
$$
where
\begin{eqnarray*}
\wedge^{1,0}(M)&:=&\{\alpha +eK^*\alpha\,\,\vert\,\,\alpha\in T^*M\}\,,\\
\wedge^{0,1}(M)&:=&\{\alpha -eK^*\alpha\,\,\vert\,\,\alpha\in T^*M\}\,,
\end{eqnarray*}
are eigenspaces for $K$ with eigenvalues $\pm e$.
We denote by
$$
dz^\alpha=dx^\alpha + edy^\alpha \quad\hbox{\rm and}\quad d\overline{z}^\alpha=dx^\alpha - edy^\alpha
$$
the basis of $\wedge^{1,0}(M)$ and $\wedge^{0,1}(M)$ dual to the bases
$\frac{\partial}{\partial z^\alpha}$ and $\frac{\partial}{\partial \overline{z}^\alpha}$, respectively. The last decomposition
induces a splitting of the bundle $\wedge^r(T^{C})^*M$ of para-complex $r$-forms on $(M,K)$ given by
$$
\wedge^r(T^{C})^*M =\bigoplus_{p+q=r}\wedge^{p,q}(M)\,.
$$
The sections of $\wedge^{p,q}(M)$ are called {\bf $(p,q)$-forms} on the para-complex manifold $(M,K)$. One can check that
\begin{equation}\label{hermitianform}
\wedge^{1,1}_{+\,-}(M)=\{\omega\in\wedge^{1,1}(M)\,\,\vert\,\,\omega=\overline{\omega}\}\,.
\end{equation}
The exterior derivative $d:\wedge^rT^*M^C\to \wedge^{r+1}T^*M^C$ splits as $d=\partial + \overline{\partial}$, where
\begin{eqnarray*}
\partial &=& \hbox{\rm pr}_{\wedge^{p+1,q}(M)}\circ d :\wedge^{p,q}(M)\to\wedge^{p+1,q}(M)\,,\\
\overline{\partial} &=& \hbox{\rm pr}_{\wedge^{p,q+1}(M)}\circ d :\wedge^{p,q}(M)\to\wedge^{p,q+1}(M)\,,
\end{eqnarray*}
and moreover, since $d^2=0$, we easily get
$$
\partial^2=0\,,\quad \overline{\partial}^2=0\,,\,\quad\partial\overline{\partial}+\overline{\partial}\partial=0\,.
$$
The operators $\partial, \overline{\partial}$ are related to $\partial_+, \partial_-$ by
$$
\partial =\frac12((\partial_++\partial_-) + e (\partial_+-\partial_-))\,,\quad
\overline{\partial} =\frac12((\partial_++\partial_-) - e (\partial_+-\partial_-))\,.
$$
In particular,
$$
\partial\overline{\partial} = e\, \partial_+\partial_-\,.
$$
We need the following result which is a consequence of a version of the Dolbeault Lemma for a para-complex manifold
(see \cite{CMMS1}).
\begin{prop}\label{Dolbeault} Let $(M,K)$ be a para-complex manifold and $\omega$ a closed $2$-form belonging to
$\wedge^{1,1}_{+\,-}(M)$. Then locally there exists a real-valued function $F$ (called {\bf potential}) such that
$$
\omega = \partial_+\partial_- F= e\,\partial\overline{\partial}F\,.
$$
The potential $F$ is defined up to addition of a function $f$ satisfying the condition $\partial_+\partial_- f=0$.
\end{prop}
\section{Para-K\"ahler manifolds.}
\subsection{Para-K\"ahler structures and para-K\"ahler potential.}
We recall three equivalent definitions of a para-K\"ahler manifold.
\begin{definition} A {\bf para-K\"ahler manifold} is given equivalently by:
\begin{enumerate}
\item[i)] a pseudo-Riemannian manifold $(M,g)$ together with a skew-symmetric
para-complex structure $K$ which is parallel with respect to the Levi-Civita connection;
\item[ii)] a symplectic manifold $(M,\omega)$ together with two complementary
involutive Lagrangian distributions $T^\pm M$.
\item[iii)] a para-complex manifold $(M,K)$ together with a symplectic form $\omega$
which belongs to $\wedge^{1,1}_{+\,-}(M)$.
\end{enumerate}
\end{definition}
The relations between the three definitions are the following. A pair $(g,K)$ as in $i)$ defines a symplectic form
$\omega = g \circ K$ and complementary involutive Lagrangian distributions $T^\pm M$
which are eigenspace distributions of $K$. One can check that the para-complex extension of the symplectic
form belongs to $\wedge^{1,1}_{+\,-}(M)$. Assume now that $(K,\omega) $ is as in iii). Then
 $g=\omega\circ K$ is a pseudo-Riemannian metric on $M$ and $(g,K)$ satisfies $i)$
 due to the following formula (see \cite[Theorem 1]{CMMS1})
$$
2g((\nabla_XK)Y,Z)=d\omega(X,Y,Z)+d\omega(X,KY,KZ)-g(N_K(Y,Z),KX)\,.
$$

Let $(M,K,\omega,g)$ be a para-K\"ahler manifold. We denote by the same letters
$g$ and $\omega$ the extensions of $g$ and $\omega$ to $C$-bilinear forms.
\begin{lemma}\label{isotropic} The following formulas hold:
\begin{itemize}
\item[i)] $g(\overline{Z},\overline{W})=\overline{g(Z,W)}\,,\qquad\quad$\ \
$\omega(\overline{Z},\overline{W})=\overline{\omega(Z,W)}\,,\,\,\,$\ $\forall\, Z,W\in T^CM$;\smallskip
\item[ii)] $g (KZ,KW)=-g(Z,W)$, $\;\;\;\omega (KZ,KW)=-\omega(Z,W)$;\smallskip
\item[iii)] the restrictions of $g$ and $\omega$ to $T^{1,0}M$ and $T^{0,1}M$ vanish.
\end{itemize}
\end{lemma}
\begin{proof} i) follows from reality of $g$ and $\omega$. ii) follows from the definition. For iii), since any element
of $T^{1,0}M$ has the form $X+eKX$, for $X\in TM$, we get
\begin{eqnarray*}
g(X+eKX,Y+eKY)&=&g(X,Y)+g(KX,KY)+ \\
&{}&+\, e\left(g(X,KY)+g(KX,Y)\right)=0\,.
\end{eqnarray*}
\end{proof}
Let now $(z^1,\ldots ,z^n)$ be a local para-holomorphic coordinate system. We denote by
$$
\partial_\a=\frac{\partial}{\partial z^\a}\,,\quad \partial_{\overline\a}=\overline{\partial}_\a =
\frac{\partial}{\partial z^{\overline\a}}\,,\quad \a=1,\ldots ,n
$$
the para-holomorphic and para-anti-holomorphic vector fields. Then we put
$$
g_{\a\overline{\b}}:=g(\partial_\a,\partial_{\overline\b})=\overline{g_{\overline{\a}\b}}
$$
and remark that
$$
g_{\a\b}=g(\partial_\a,\partial_\b)=0\,,\quad g_{\overline{\a}\overline{\b}}=g(\partial_{\overline{\a}},\partial_{\overline{\b}})=0\,.
$$
In these coordinates,
$$
g=g_{\a\overline{\b}}dz^\a d\overline{z}^\b + \overline{g_{\a\overline{\b}}dz^\a d\overline{z}^\b}\,,\quad
\omega =2\sum_{\a\,,\b}\omega_{\a\overline\b}dz^\a\wedge d\overline{z}^{\b}\,,
$$
where $\omega_{\a\overline\b}=e\,g_{\a\overline\b}$. Since $\omega$ is closed, we have
$$
\frac{\partial g_{\a\overline\b}}{\partial z^\c}=\frac{\partial g_{\c\overline\b}}{\partial z^\a}\,,\quad
\frac{\partial g_{\a\overline\b}}{\partial \overline{z}^{\c}}=\frac{\partial g_{\a\overline\c}}{\partial\overline{z}^{\b}}\,.
$$
Proposition \ref{Dolbeault} implies the local existence of a real function $F$ such that
$$
g_{\a\overline{\b}}=\partial_\a\partial_{\overline{\b}}F\,.
$$
The function $F$ is called the {\bf para-K\"ahler potential} of the para-K\"ahler metric $g$.

\subsection{Curvature tensor of a para-K\"ahler metric.}
Denote by $\nabla$ the Levi-Civita connection of
the pseudo-Riemannian metric $g$ and by $\Gamma^C_{AB}= \Gamma^C_{BA}$ the Christoffel symbols
with respect to a para-holomorphic coordinates, where $A,B,C$ denote both Greek indices and
their conjugates.

\begin{lemma}
The only possible non-zero Christoffel symbols are
$$
\Gamma^\c_{\a\b}=\Gamma^{\overline{\c}}_{\overline{\a}\overline {\b}}\,.
$$
\end{lemma}
\begin{proof} The condition $\nabla K=0$ implies
$$
\nabla_{\partial_A}K\partial_B-K\nabla_{\partial_A}\partial_B=0\,.
$$
Hence,
\begin{eqnarray*}
e\nabla_{\partial_\a}\partial_\b-K\left(\sum_{\c}\Gamma_{\a\b}^{\c}\partial_\c+ \sum_{\c}\Gamma_{\a\b}^{\overline\c}\overline{\partial}_\c\right)
&=&0\,,\\
e\left(\sum_{\c}\Gamma_{\a\b}^{\c}\partial_\c+ \sum_{\c}\Gamma_{\a\b}^{\overline\c}\overline{\partial}_\c\right)-
 e\sum_{\c}\Gamma_{\a\b}^{\c}\partial_\c+ e\sum_{\c}\Gamma_{\a\b}^{\overline\c}\overline{\partial}_\c&=&0\,,\\
2e \sum_{\c}\Gamma_{\a\b}^{\overline\c}\overline{\partial}_\c&=&0\,,
\end{eqnarray*}
which implies $\Gamma_{\a\b}^{\overline\c}=0$. The other computations are similar.
\end{proof}
By the formula relating the Levi-Civita connection to the metric, we can express the Christoffel symbols
by
\begin{equation}\label{christoffel1}
\sum_\a g_{\a\overline{\m}}\Gamma^\a_{\b\c}=\frac{\partial g_{\overline{\m}\b}}{\partial z^\c}\,,\qquad
\sum_\a g_{\overline{\a}\m}\Gamma^{\overline{\a}}_{\b\overline{\c}}=\frac{\partial g_{\m\overline{\b}}}{\partial \overline{z}^\c}\,.
\end{equation}

\begin{prop}\label{proprietacurvatura}
The curvature tensor $R$ and the Ricci tensor $S$ of a para-K\"ahler metric $g$ satisfy the following
relations
\begin{eqnarray}\label{R}
&{}& R(X,Y)\circ K = K\circ R(X,Y)\,,\quad R(KX,KY)=-R(X,Y)\,,\label{formulaR}\\
&{}& S(KX,KY)=-S(X,Y)\label{formulaS}\,,
\end{eqnarray}
for any vector fields $X,Y\in \mathfrak{X}(M)$.
\end{prop}
\begin{proof} Since $K$ is parallel with respect to the Levi-Civita connection $\nabla$ of the para-K\"ahler metric
$g$, we easily get that $R(X,Y)$
and $K$ commute for any $X$, $Y$. Hence the first formula of \eqref{formulaR} is proved.\newline
We have:
\begin{eqnarray*}
g(R(KX,KY)V,U) &=& g(R(U,V)KY,KX)\\
&=& g(KR(U,V)Y,KX)\\
&=& -g(R(U,V)Y,X)\\
&=& -g(R(X,Y)V,U)\,,
\end{eqnarray*}
which implies that $R(KX,KY) = - R(X,Y)$.\newline
By definition of the Ricci tensor $S$, we get
\begin{eqnarray*}
S(KX,KY) &=& \tr \left(V\mapsto R(V,KX)KY)\right)\\
&=& -\tr \left(KV\mapsto R(KV,KX)KY)\right)\\
&=& \tr \left(KV\mapsto R(V,X)KY)\right)\\
&=& \tr \left(KV\mapsto KR(V,X)Y)\right)\\
&=& -\tr \left(V\mapsto R(V,X)Y)\right)\\
&=& -S(X,Y)\,,
\end{eqnarray*}
where we used \eqref{formulaR} and the fact that $g(KX,KY)=-g(X,Y)$.
\end{proof}

\begin{prop}\label{curvatura}
The only possible non-zero components of the Riemann curvature tensor R are
\begin{eqnarray*}
&{}& R^\a_{\b\gamma\overline{\d}}\,,\quad
R^\a_{\b\overline{\gamma}\d}\,,\quad
R^{\overline\a}_{\overline{\b}\gamma\d}\,,\quad
R^{\overline\a}_{\b\overline{\gamma}\d}\,,
\\
&{}&
R_{\a\overline{\b}\c\overline{\d}}\,,\quad
R_{\a\overline{\b}\overline{\c}\d}\,,\quad
R_{\overline{\a}\b\c\overline{\d}}\,,\quad
R_{\overline{\a}\b\overline{\c}\d}\,.
\end{eqnarray*}
\end{prop}
\begin{proof} Since $R(\partial_C,\partial_D)\circ K= K\circ R(\partial_C,\partial_D)$, we have
\begin{eqnarray*}
g(R(\partial_C,\partial_D)\overline{\partial}_\b,\overline{\partial}_\c)&=& -e\,g(R(\partial_C,\partial_D)K\overline{\partial}_\b,
\overline{\partial}_\c)\\
&=& -e\,g(K R(\partial_C,\partial_D)\overline{\partial}_\b,\overline{\partial}_\c) \\
&=&e\,g( R(\partial_C,\partial_D)\overline{\partial}_\b,K\overline{\partial}_\c)\\
&=&-\,g(R(\partial_C,\partial_D)\overline{\partial}_\b,\overline{\partial}_\c)\,.
\end{eqnarray*}
Hence, $R^\a_{\overline{\b}CD}=0$. In a similar way, taking into account the symmetry properties
of the Riemann tensor, the statement can be proved.
\end{proof}
By the formulas above, recalling the expression of $R$ in terms of $\Gamma^A_{BC}$, it follows that
\begin{equation}\label{Rgamma}
R^\a_{\b\c\overline{\d}}=-\frac{\partial \Gamma^{\a}_{\b\c}}{\partial \overline{z}^\d}\,.
\end{equation}

\subsection{The Ricci form in para-holomorphic coordinates.}
The Ricci tensor of the metric $g$ is defined by
$$
\Ric_{AB}=\sum_{C}R^C_{ACB}\,.
$$
Therefore, by Proposition \ref{curvatura} and \eqref{Rgamma}, we obtain
\begin{equation}\label{curvaturaricci}
\Ric_{\a\overline{\b}}=-\sum_\c \frac{\partial \Gamma^\c_{\a\gamma}}{\partial \overline{z}^\b}\,,\qquad
\Ric_{\overline{\a}\b}=\overline{R}_{\a\overline{\b}}\,,\qquad
\Ric_{\a\b}=\Ric_{\overline{\a}\overline{\b}}=0\,.
\end{equation}
We define the {\bf Ricci form $\rho$} of the para-K\"ahler metric $g$ by
\begin{equation}\label{formaricci}
\rho:=\Ric\circ K\,.
\end{equation}
We extend it to a para-complex $2$-form $\rho$. Formula
\eqref{curvaturaricci} shows that $\rho$ has type $(1,1)$ and in local coordinates can be represented as
$$
\rho =2e\,\Ric_{\a\overline{\b}}dz^\a\wedge d\overline{z}^\b\,.
$$
\begin{prop} The Ricci form of a para-K\"ahler manifold is a closed $(1,1)$-form and can be represented by
\begin{equation}\label{ricci}
\rho = e\,\partial\overline{\partial}\log(\det(g_{\a\overline{\b}}))\,.
\end{equation}
In particular,
\begin{equation}\label{tensorericci}
\Ric_{\a\overline{\b}}=-\frac{\partial^2 \log(\det(g_{\lambda\overline{\m}}))}{\partial z^\a\partial \overline{z}^\b}\,.
\end{equation}
\end{prop}
\begin{proof} We remark that the metric $g$ defines a $C$-linear bijective
map $g: T^{1,0}M\to T^{0,1*}M$. There exists inverse map $g^{-1} :T^{0,1*}M\to T^{1,0}M$ which is $C$-linear. It is represented by
a matrix $(g^{\overline{\b}\a})$ such that $g_{\a\overline{\b}}g^{\overline{\b}\gamma}=\delta^\gamma_\a$.\newline
Using \eqref{christoffel1} and the identity that is still valid in para-complex case
\begin{equation}\label{determinante}
\frac{\partial\det(g_{\lambda\overline{\m}})}{\partial z^\a}
=\det(g_{\lambda\overline{\m}})\sum_{\b\c}g^{\b\overline{\c}}\frac{\partial g{_{\b\overline{\c}}}}{\partial z^\a}\,,
\end{equation}
we obtain
\begin{equation}\label{christoffel2}
\Gamma^\a_{\b\c}=\sum_\m g^{\a\overline{\m}}\frac{\partial g_{\overline{\m}\b}}{\partial z^\c}\,.
\end{equation}
Hence, by \eqref{determinante} and \eqref{christoffel2} we get
$$
\sum_\c \Gamma^\c_{\a\c} =\frac{\partial \log(\det(g_{\lambda\overline{\m}}))}{\partial z^\a}\,.
$$
The last formula implies
\begin{equation}\label{tensorericci1}
\Ric_{\a\overline{\b}}=\frac{\partial^2 \log(\det(g_{\lambda\overline{\m}}))}{\partial z^\a\partial \overline{z}^\b}\,,
\end{equation}
which proves that
$$
\rho = e\,\partial\overline{\partial}\log(\det(g_{\a\overline{\b}}))\,.
$$
\end{proof}

\subsection{The canonical form of a para-complex manifold with volume form. }
Let $(M,K,\vol)$ be an oriented manifold with para-complex structure $K$ and
a (real) volume form $\vol$. We define a canonical $(1,1)$-form $\rho$ on $M$.
Let $z=(z^1,\ldots ,z^n)$ be local para-holomorphic coordinates
and $(x^\alpha,y^\alpha)$ corresponding real coordinates, where
$z^\alpha=x^\alpha + ey^\alpha$.
\newline
Then we can write
\begin{eqnarray}\label{volumeform}
\vol &=& V(z,\overline{z}) dz^1\wedge d\overline{z}^1 \wedge \ldots \wedge dz^n\wedge d\overline{z}^n \\
&=& U(x,y) dx^1\wedge d y^1 \wedge \ldots \wedge dx^n\wedge dy^n\,.\nonumber
\end{eqnarray}
We may assume that $U(x,y) >0$, as $M$ is oriented.\newline
Since
$$
dz^\alpha\wedge d\overline{z}^\alpha = (dx^\alpha +edy^\alpha)\wedge (dx^\alpha -edy^\alpha)= -2e \,dx^\alpha\wedge dy^\alpha
$$
then
$$
(-2e)^n V(z,\overline{z})= U(x,y)\,.
$$
In particular the, function $(-e)^n V$ is positive.\newline
Let $z'^\alpha(z)=x'^\alpha(z)+e y'^\alpha(z)$ be another para-holomorphic coordinates such that the
associated real coordinates $(x'^\alpha,y'^\alpha)$ have the same orientation and
$V'(z',\overline{z'}), U'(x',y')$ be corresponding functions as in \eqref{volumeform}. Then
$$
V'(z',\overline{z'})=V(z,\overline{z})\,\Delta(z)\,\overline{\Delta(z)}\,,\quad
U'(x',y') = U(x,y)\, J(x,y)
$$
where $\Delta(z) =\det\Vert\frac{\partial z'}{\partial z}\Vert$ and $J=\det\Vert\frac{\partial(x',y')}{\partial (x,y)}\Vert >0$
are the Jacobian of the corresponding transition functions. Since $(-2e)^nV' = U'$ we have
$$
\Delta(z)\,\overline{\Delta(z)} = J(x,y) >0\,.
$$
If we write
$$
\Delta(z)=u(z) + ev(z)\,,
$$
then
$$
\Delta(z)\,\overline{\Delta(z)} =u^2-v^2 = J >0\,.
$$
Hence $u\neq 0$. This implies the following
\begin{lemma} The formula
\begin{equation}\label{canonicalform}
\rho =e\,\partial \overline{\partial} \log \left((-e)^nV\right)
\end{equation}
defines a real global closed $2$-form of type $(1,1)$ on the oriented para-complex manifold $(M,K,\vol)$.
\end{lemma}
The form $\rho$ is called the {\bf canonical form} on $(M,K,\vol)$.
\begin{proof} Since $(-e)^nV(z,\overline{z})$ is a positive smooth function of
$(z,\overline{z})$, logarithm of $(-e)^nV$ is a well defined
smooth function and $\rho =\partial \overline{\partial} \log \left((-e)^nV\right)$ is a $(1,1)$-form. It remains to check that if
$z'$ is another
coordinate system and $V'=V(z,\overline{z})\,\Delta(z)\,\overline{\Delta(z)}$ is the associated function, then \newline
$$
\rho' =\partial \overline{\partial} \log\left( (-e)^nV'\right)= \rho\,.
$$
Since the real part $u$ of $\Delta(z) =u+ e v$ is not zero, we can choose $\epsilon =\pm 1$ such that
$\Delta^\epsilon := \epsilon \Delta$ and $\overline{\Delta^\epsilon}:=\epsilon \overline{\Delta}$
has positive real part, then the function $\log ( \Delta^\epsilon(z))$ is a para-holomorphic function.
In particular, $\overline{\partial}\log(\epsilon \Delta(z))=0$. Similarly, $\log ( \overline{\Delta^\epsilon(z))}$ is
anti-para-holomorphic. Then
\begin{eqnarray*}
\rho' &=&e\,\partial \overline{\partial} \log \left((-e)^nV'\right) = e\,\partial \overline{\partial}
\log \left((-e)^nV\,\Delta^\epsilon(z)\overline{\Delta^\epsilon(z)}\right)\\
&=&e\,\partial \overline{\partial} \left[\log\left((-e)^nV\right)+\log\left(\Delta^\epsilon(z)\right)+
\log\left(\overline{\Delta^\epsilon(z)}\right) \right]\\
&=&e\,\partial \overline{\partial} \log\left((-e)^nV\right) =\rho\,.
\end{eqnarray*}
\end{proof}
Hence, from formulas \eqref{ricci} and \eqref{canonicalform}, we get the following
\begin{cor} Let $(M,K,\omega,g)$ be an oriented para-K\"ahler manifold and $\vol^g$ the volume form associated with the
metric $g$. Then the Ricci form $\rho$ of the para-K\"ahler manifold
$M$ coincides with the canonical form of $(M,K,vol^g)$. In
particular $\rho$ depends only on the para-complex structure and the
volume form.
\end{cor}
\medskip
Now we derive a formula for the canonical form $\rho$ in term of divergence.\newline
Let $(M,K)$ be a $2n$-manifold with a para-complex structure and $\vol$
 be a (real) volume form on $M$. With respect to the local
 para-holomorphic coordinates $z^\alpha$ on $M$, we can write
 $$
 \vol = V(z,\bar z)dz^1 \wedge \ldots\wedge dz^n \wedge d \bar z^1
 \wedge \ldots \wedge d \bar z^n\,,
 $$
 where $V$ is a real or imaginary valued function depending on the parity of $n$. \par
 We define the {\bf divergence} $\diver X$
 of a vector field $X \in \mathfrak X (M)$ on $(M, vol)$ by
$$
\mathcal{L}_X \,\vol = (\diver X) \vol\,,
$$
where $\mathcal{L}_X$ denotes the Lie derivative along the vector field $X$.
Then for $X,Y \in \mathfrak X(M)$ and any smooth function $f \in \mathcal C^\infty(M)$ we have
\begin{equation}\label{proprietadivergenza}
X (\diver Y) - Y(\diver X) = \diver [X,Y]\,,\quad
\diver(fX)= f \diver X + X(f).
\end{equation}
Moreover, setting $\partial_i =\frac{\partial}{\partial x^i}$, we have
$$\diver(\partial_i) = \partial_i \log ((-e)^nV)$$ and if $X =
X^i
\partial_i$ then
\begin{equation}\label{divergenza}
\diver X = \partial_i X^i + \partial_i (\log ((-e)^nV)) X^i\,.
\end{equation}
For any
$$
Z = Z^\alpha \partial_\alpha + Z^{ \bar \alpha} \partial_{\bar \alpha}\,,\quad W = W^\beta \partial_\beta + W^{ \bar \beta}
\partial_{\bar \beta}\,,
$$
we denote by
$$
h(Z,W) = \partial_\alpha \partial_{\bar \beta} \log ((-e)^nV) Z^\alpha W^{\bar \beta} +
\partial_\alpha \partial_{\bar \beta} \log ((-e)^nV) W^\alpha Z^{\bar \beta}
$$
the para-Hermitian form associated with $\rho$.
Then
$$
\rho(Z,W) = h(Z,KW)\,.
$$
The following lemma will be used in the next section
\begin{lemma}\label{rho}
Let $X, Y$ be vector fields on $M$ such that $\diver X =\diver
Y=0$ and $\mathcal{L}_XK=\mathcal{L}_YK=0$, where $\mathcal{L}$ denotes the Lie derivative. Then
\begin{equation}\label{quattro}
2\rho (X,Y) =\diver(K[X,Y])\,.
\end{equation}
\end{lemma}
\begin{proof} Set
$$
X^c = X + eKX\,,\,\,\, Y^c = Y + eKY\,;
$$
then, $X^c$ and $Y^c$ are para-holomorphic vector fields and by definition of $\rho$ and $h$ we have
\begin{eqnarray*}
2\rho(X,Y) & = &\frac{e}{2} \left( -h(X^c,\overline{Y^c})+h(\overline{X^c},Y^c) \right)\\
&=& \frac{e}{2}\left(-h(X^c,\overline{Y^c})+\overline{h(X^c,\overline{Y^c})} \right)\\
& =& -2\im\, h(X^c,\overline{Y^c})\\
&=& -\im\, (X^c(\diver\overline{Y^c}))\\
&=& (X(\diver KY)- K X(\diver Y))\\
&=& \,\diver ([X,KY])\\
&=& \,\diver (K[X,Y])\,,
\end{eqnarray*}
where in the last equality we used that $\mathcal{L}_XK=0$.
\end{proof}

\section{Homogeneous para-K\"ahler manifolds.}
\subsection{The Koszul formula for the canonical form.}
Let $M = G/H$ be a homogeneous reductive manifold
with an invariant volume form $\vol$. We fix a reductive decomposition $\ggg = \gh + \gm$
of the Lie algebra $\ggg$ and identify the subspace $\gm$ with the tangent space $T_oM$ at
the point $o = eH$.\\
We denote
by $\Omega = \pi^* \vol$ the pull back of $\vol$ under the natural
projection $\pi : G \rightarrow G/H$. We choose a basis $\omega^i$
of the space of left-invariant horizontal 1-forms on $G$ such
that $\Omega = \omega^1 \wedge \ldots \wedge \omega^m$ and denote by $X_i$ the dual
basis of horizontal left-invariant vector fields. (The
horizontality means that $\omega^i$ vanish on fibers of $\pi$ and the vector fields $X_i$
tangent to left-invariant distribution
generated by $\mathfrak{m}$). For any element $X \in \mathfrak{g}$ we denote by the
same letter $X$ the corresponding left-invariant vector field on
$G$ and by $X'$ the corresponding right-invariant vector field on $G$.
\begin{lemma}(Koszul \cite{K})\label{koszul} Let $X$ be the velocity vector field on $M = G/H$
of a 1-parameter subgroup of $G$. Then the pull-back
$\pi^*(\diver X)$ of the divergence $\diver X$ is given by
$$
\pi^* (\diver X) = \sum_{i=1}^{2n} \omega^i ([X_i,X])\,.
$$
\end{lemma}
\begin{proof} For any projectable vector field $\tilde X$ on $G$ with
the projection $X$ we have
$$
\pi^* (\diver X) \Omega=\pi^* (\mathcal{L}_X \,vol ) = \mathcal{L}_{\tilde X} \, \Omega =
\sum \omega^i ([X_i,X])\Omega
\,.
$$
\end{proof}
Now we assume that $G/H$ is endowed with an invariant para-complex
structure $K$ and we derive a formula for the canonical form. We
extend the endomorphism $K|_{T_oM} = K|_{\gm}$ to the
${\Ad}_H$-invariant endomorphism $\tilde{K}$ of $\ggg$ with
kernel $\gh$ and we denote by the same symbol the associated
left-invariant field of endomorphisms on the group $G$.

\begin{prop}\label{propkoszul} Let $M = G/H$ be a homogeneous manifold with an
invariant volume form $\vol$ and an invariant para-complex
structure $K$. Then the pull-back $\pi^* \rho$ to $G$ of
the canonical $2$-form $\rho$ associated with $(\vol, K)$ at the
point $o = eH$ is given by
$$
2 \pi^*\rho (X,Y) = \sum \omega^i \left([\tilde{K}[X,Y],X_i]- \tilde{K}[[X,Y], X_i]\right)\,,\quad \forall\,X,Y \in \ggg\,.
$$
In particular,
$$
2\pi^*\rho_e =d\psi\,,
$$
where $\psi$ is the $\ad_{\mathfrak{h}}$-invariant $1$-form on $\ggg$ given by
\begin{equation}\label{koszulform1}
\psi(X) =-\tr_{\mathfrak{g}/\gh} \left(\hbox{\rm ad}_{\tilde{K}X} - \tilde{K} \hbox{\rm ad}_X\right)\,,\quad \forall\,X \in \ggg\,.
\end{equation}
\end{prop}
The $1$-form $\psi$ (and the associated left-invariant $1$-form on $G$) is called the {\bf Koszul form}.
\begin{proof}

 We denote by $X',Y', Z' = [X',Y']$ right-invariant vector fields on $G$ associated with elements
 $X,Y,Z = [X,Y]$ of $\ggg$. These vector fields project to vector fields on $M = G/H$ which are generators
 of 1-parameter subgroups of $G$ acting on $M$; in particular, they preserve the para-complex structure and the
 volume form.
Applying \eqref{quattro} of Lemma \ref{rho} and Lemma \ref{koszul} we get
\begin{eqnarray*}
2\pi^*\rho(X,Y)&=& 2\rho (\pi_*X',\pi_*Y')_e =\diver (K\pi_*[X',Y'])_e \\
&=& \diver(\pi_*\tilde{K}Z')_e = \sum_i\omega^i\left([X_i,\tilde{K}Z'] \right)_e \,.
\end{eqnarray*}

Since left-invariant vector fields $X_i$ commute with right-invariant vector field $Z'$,we can write
\begin{eqnarray*}
[X_i,\tilde{K}Z']_e &=& (\mathcal L_{X_i}\tilde K )Z'_e =({\mathcal
L}_{X_i}\tilde K )Z|_e \\
& =&{\mathcal L}_{X_i}(\tilde K Z)|_e -\tilde K ({\mathcal L}_{X_i}Z)|_e\\
&=&[X_i, \tilde K Z] - \tilde K ([X_i, Z])\,.
\end{eqnarray*}
 Here we consider $Z$ and $\tilde K Z$ as left-invariant vector fields on $G$.
 Substituting this formula into the previous one, we obtain the statement.
\end{proof}

\subsection{Invariant para-complex structures on a homogeneous manifold.}
 Let $M=G/H$ be a homogeneous manifold and $\mathfrak{h}, \mathfrak{g}$ be the
 Lie algebras of $H,G$, respectively.
 \begin{prop} There is a natural 1-1 correspondence between invariant para-complex
 structures $K$ on $M$ with eigenspaces decomposition $TM = T^+ \oplus T^-$ and decompositions
 $\mathfrak{g}= \mathfrak{g}^+ + \mathfrak{g}^-$ into two ${\mathrm Ad}_H$-invariant
 subalgebras $\mathfrak{g}^\pm$ with $\mathfrak{g}^+ \cap \mathfrak{g}^- = \mathfrak{h}$
 and $\dim{\mathfrak{g}^+}/ \mathfrak{h} = \dim{\mathfrak{g}^-}/ \mathfrak{h}$.
 \end{prop}
 \begin{proof} Indeed, such a decomposition defines two complementary ${\mathrm Ad}_H$-invariant subspaces
 $\mathfrak{g}^\pm/\mathfrak{h} \subset \mathfrak{g}/\mathfrak{h} = T_oM$ which can be extended to
 complementary invariant integrable distributions $T^\pm$.
 \end{proof}

 An invariant para-complex structure $K$ associated with a graded Lie algebra can be constructed as follows.
 Let
 \begin{equation}\label{gradation}
\ggg = \ggg_{-k} +\cdots + \ggg_{-1}+\ggg_{0}+\ggg_{1}+\cdots + \ggg_{k}\,,\,\,\,[\ggg_i,\ggg_j]\subset\ggg_{i+j}
\end{equation}
be a $\mathbb{Z}$-graded Lie algebra. Then the endomorphism $D $ defined by $D|_{\mathfrak{g}_j} = j \Id$
is a derivation of $\mathfrak{g}$. Extending the Lie algebra $\mathfrak{g}$ to $\mathbb{R}D + \mathfrak{g}$ if necessary, we may assume
that the derivation $D$ is inner, i.e. there is an element $d \in \mathfrak{g}_0$, called {\bf grading element}, such that
$D = \ad_d$. Let $G$ be a connected Lie group with Lie algebra $\mathfrak{g}$ and $H$ be the (automatically closed)
subgroup generated by $\mathfrak{h}:= \mathfrak{g}_0$. Then the decomposition
$$
\mathfrak{g} = \mathfrak{g}^+ + \mathfrak{g}^-\,, \quad \mathfrak{g}^\pm = \mathfrak{g}_0 + \sum_{j>0} \mathfrak{g}_{\pm j}
$$
 defines an invariant para-complex structure $K$ on the (reductive) homogeneous manifold $M = G/H$. The homogeneous para-complex
 manifold $(M = G/H, K)$ and the para-complex structure $K$ are called the {\bf para-complex manifold and the para-complex
 structure associated with a gradation}.

 Note that the grading element $d$ generates an Anosov flow
 $ \varphi_t:= {\exp}(td) $ on $M = G/H$ with smooth stable distribution $T^-$ and unstable distribution
 $T^+$ which preserves the canonical invariant connection on $M$ associated with the reductive decomposition
 $\mathfrak{g} = \mathfrak{h}+ \mathfrak{m}: = \mathfrak{g}_0 + (\sum_{j \neq 0}\mathfrak{g}_j )$.
 The result by Benoist and Labourie \cite{BL}, who describes Anosov diffeomorphisms on a compact manifold
 with smooth stable and unstable distributions which preserve a linear connection (or symplectic form), shows that
 there is no compact smooth quotient of $M$ which preserves the Anosov diffeomorphism $\varphi_1$, that is a cocompact
 freely acting discrete group $\Gamma$ of diffeomorphisms of $M$ which commutes with the Anosov diffeomorphism $\varphi_1$.
 Moreover, the following deep result \cite{BL1} shows that there is no compact manifold modelled on the homogeneous manifold
 $M=G/H$ of a semisimple Lie group $G$ associated with a graded Lie algebra. Recall that a manifold $N$ is modeled  on a homogeneous
 manifold $G/H$ if there is an atlas of $G/H$-valued local charts whose transition functions are restrictions of elements of $G$.

 \begin{thm} \cite{BL1} Let $G/H$ be a homogeneous manifold of a connected semisimple Lie group $G$ which admits an
 invariant volume form. If there is a compact manifold $N$ modelled on $G/H$, then $H$ does not contain any hyperbolic element.
 In particular, if $G/H$ admits a smooth compact quotient $\Gamma \setminus G/H$, then the center of $H$ is compact.
 \end{thm}
Note that a semisimple group $G$ admits a cocompact discrete subgroup $\Gamma$ such that the quotient $\Gamma \setminus G/H$
is a compact orbifold and it would be interesting to construct such an orbifold with induced Anosov diffeomorphism.
\subsection{Invariant para-K\"ahler structures on a homogeneous reductive manifold.}
 Now we give a Lie algebraic description of invariant para-K\"ahler structures on
 a homogeneous manifold $M = G/H$ with a reductive decomposition
 $ \mathfrak{g} = \mathfrak{h} + \mathfrak{m, \,\, [\mathfrak{h}, \mathfrak{m}] \subset \mathfrak{m}}$.
 We denote by
 $$j : H \rightarrow GL(\mathfrak{m}),\,\, h \mapsto j(h) = \Ad_h|_{\mathfrak{m}}$$
 the isotropy representation of $H$ into $\mathfrak{m} = T_oM$.
Recall that an invariant symplectic structure on $M = G/H$ is defined by a closed
$Ad_H$-invariant $2$-form $\omega$ on $\ggg$ with kernel $\gh$.
The invariant para-complex structure $K$ associated with a decomposition
$$ \mathfrak{g} = \mathfrak{g}^+ + \mathfrak{g}^- = (\mathfrak{h}+ \mathfrak{m}^+ )+ ( \mathfrak{h}+ \mathfrak{m}^-)$$
 is skew-symmetric with respect
to the invariant symplectic form $\omega$ (in other words $\omega$
is of type $(1,1)$) if and only if $\omega\vert_{\gm^\pm}=0$ (see
lemma \ref{isotropic}). This implies
\begin{prop} Invariant para-K\"ahler structures on a reductive homogeneous manifold $M=G/H$ are defined by triples
$(\omega,\gm^+,\gm^-)$, where $\gm = \gm^+ +\gm^-$ is a $j(H)$-invariant decomposition such that $\ggg^\pm= \gh +\gm^\pm$ are
subalgebra of $\ggg$ and $\omega$ is a closed $Ad_H$-invariant $2$-form with kernel $\gh$ such that $\omega\vert_{\gm^\pm}=0$.
\end{prop}

Note that an invariant volume form on a homogeneous manifold $M=G/H$ exists if and only if the isotropy representation
$j(H)$ is unimodular (i.e. $\det (j(h))=1$ for all $h\in H$), and it is defined up to a constant scaling.
We get
the following
\begin{cor}\label{corhomogparakaehler} Let $(M=G/H,K)$ be a homogeneous para-complex manifold which admits an invariant volume form
$\vol$. Then any invariant para-K\"ahler
structure $(K,\omega)$ has the same Ricci form $\rho$ which is the canonical form of $(K,\vol)$. Moreover, there
exists an invariant para-K\"ahler Einstein structure with non zero scalar curvature if and only if the canonical form
$\rho$ is non degenerate. These structures are given by pairs $(K,\omega)$, with $\omega=\lambda\rho$, $\lambda\neq 0$.
\end{cor}

\section{Homogeneous para-K\"ahler Einstein manifolds of a semisimple group.}
The aim of this section is to describe invariant para-K\"ahler
Einstein structures on homogeneous manifolds $M=G/H$ of semisimple
groups $G$.

 We need the following important result.
 \begin{thm}[\cite{HDKN}] A homogeneous manifold $M= G/H$ of a semisimple Lie group $G$ admits
 an invariant para-K\"ahler structure $(K, \omega)$ if and only if it is a covering of a semisimple
 adjoint orbit $\Ad_G h = G/ Z_G(h) $ that is the adjoint orbit of a semisimple element $h \in \ggg$.
 \end{thm}
 Note that in this case $Z^0_G(h) \subset H \subset Z_G(h)$, where $Z^0_G(h)$ denote the connected centralizer of $h$ in $G$ and the element
 $h$ is $\gh$-regular, i.e. its centralizer in $\ggg$ is
 $Z_{\ggg}(h) = \gh$.

We will describe invariant para-complex structures $K$ on such a homogeneous manifold $M = G/H$
in term of Satake diagrams and invariant symplectic structures $\omega$. Then we describe
the canonical form $\rho=\rho_K$ on $(G/H,K)$ in term of roots and show that it is non degenerate. This implies
that for any invariant para-complex structure $K$ there exists a unique para-K\"ahler Einstein structure $(K,\lambda\rho_K)$ with
given non-zero scalar curvature.

\subsection{Invariant para-K\"ahler structures on a homogeneous manifold.}
Let $M=G/H$ be a covering of an adjoint orbit $\Ad_G h $ of a
real semisimple Lie group $G$ that is $Z^{0}_G(h) \subset H
\subset Z_G(h)$. Since the Killing form $B\vert_\gh$ is
non-degenerate, the $B$-orthogonal complement $\gm =\gh^\bot$ of
$\gh$ defines a reductive decomposition
$$
\ggg =\gh +\gm\,.
$$
We recall that a gradation
\begin{equation}\label{gradation}
\ggg = \ggg_{-k} +\cdots + \ggg_{-1}+\ggg_{0}+\ggg_{1}+\cdots + \ggg_{k}\,,\,\,\,[\ggg_i,\ggg_j]\subset\ggg_{i+j}
\end{equation}
of a semisimple Lie algebra $\ggg$ is defined by a grading element $d \in \mathfrak{g}_0$.
 A gradation is called {\bf fundamental}, if the subalgebras
\begin{equation}\label{mpm}
\gm_\pm =\sum_{i>0}\ggg_i\,.
\end{equation}
are generated by $\ggg_{\pm 1}$, respectively.
\begin{lemma} A gradation
is $\Ad_H$-invariant, where $H \subset G$ is a subgroup of $G$ with the Lie algebra $\gh = \ggg_0$ if and only if $\Ad_H$ preserves $d$.
\end{lemma}

The following proposition describes all invariant para-Complex structures $K$ on a homogeneous manifold $M = G/H$
of a semisimple Lie group $G$ which is a cover of an adjoint orbit of a semisimple element $h \in \mathfrak{g}$.

\begin{prop}[\cite{AM}] There is a natural $1-1$ correspondence between invariant para-complex structures $K$ on a homogeneous
manifold $G/H$ of a semisimple Lie group $G$ which is a covering of a semisimple adjoint orbit
 $M=Ad_G(h)$ and
$\Ad_H$-invariant fundamental gradations \eqref{gradation} of the
Lie algebra $\ggg$ with $\ggg_0 =\gh$. The gradation
\eqref{gradation} defines the para-complex structure $K$ with eigenspace decomposition $TM = T^+ + T^-$
where invariant integrable distributions $T^\pm$ are invariant extensions of the subspaces
 $\gm_\pm$ defined by \eqref{mpm}.
\end{prop}

 Let $\gh = Z_{\ggg}(h)$ be the centralizer of a semisimple element $h$.
We have a $B$-orthogonal decomposition $\gh = \gz + \gh'$, where
$\gz$ is the center of $\gh$ and $\gh'=[\gh,\gh]$. Recall
that an element $z\in\gz$ is called $\gh$-{\bf regular} if
$Z_\ggg(z)=\gh$.We need the following known
\begin{prop} Let $M = G/H$ be a homogeneous manifold as in the previous proposition.
Then there exists a natural $1-1$ correspondence between $\Ad_H$-invariant elements $z\in\gz$ and closed invariant $2$-forms $\omega_z$ on
$M=G/H$ given by
$$
\omega_z(X,Y)=B(z,[X,Y])\,,\quad \forall X,Y\in\gm =T_oM\subset \ggg\,.
$$
Moreover $\omega_z$ is a symplectic form if and only if $z$ is $\gh$-regular.
\end{prop}
\begin{cor} Any invariant para-complex structure $K$ on $M=G/H$ is skew-symmetric with respect to any invariant symplectic structure. In
other words, any pair $(K,\omega)$ defines an invariant para-K\"ahler structure.
\end{cor}

\subsection{Fundamental gradations of a real semisimple Lie algebra.}
We recall the description of a real form of a complex
semisimple Lie algebra in terms of Satake diagrams, which are
extensions of Dynkin diagrams (see \cite{GOV}).

Any real form of a complex semisimple Lie algebra $\ggg$ is the
fixed point set $\ggg^\sigma$ of an anti-linear involution $\sigma$ of $\ggg$.
A Cartan subalgebra $\ga^{\sigma}$ of $\ggg^{\sigma}$ decomposes into a direct sum
$\ga=\ga^+\oplus\ga^-$, where
\begin{eqnarray}
&& \ga^+ := \{X\in\gh \,|\, \ad_{\ggg}(X) \mbox{ has purely imaginary eigenvalues}\}\,,\\
&& \ga^- := \{X\in\gh \,|\, \ad_{\ggg}(X) \mbox{ has real eigenvalues}\}\,,
\end{eqnarray}
are called the {\bf toroidal} and the {\bf vectorial} part of $\ga$, respectively.

Let $\ga^\s$ be a maximal vectorial Cartan subalgebra of $\ggg^\s$,
i.e.\ such that the vectorial part $\ga^-$ has maximal dimension.
Then the root decomposition of
$\ggg^\sigma$, with respect to the subalgebra
$\ga^\sigma $, can be written as
$$
\ggg^\sigma =\ga^\sigma + \sum_{\lambda\in\Sigma}\ggg^\sigma_\lambda\,,
$$
where $\Sigma\subset (\ga^-)^*$ is a (non-reduced) root system.

Denote by $\ga =(\ga^\sigma)^{\C}$ the complexification of
$\ga^\sigma$ (which is a $\sigma$-invariant Cartan subalgebra)
and by $\sigma^*$ the induced anti-linear action of
$\sigma$ on $\ga^*$:
$$
\sigma^*\alpha =\overline{\alpha\circ\sigma}\,,\quad\alpha\in\ga^*\,.
$$
Consider the root space decomposition of $\ggg$ with respect to $\ga$:
$$
\ggg =\ga+\sum_{\alpha\in R} \ggg_\alpha \,,
$$
where $R$ is the root system of $(\ggg,\ga)$.
Note that $\sigma^*$ preserves $R$, i.e.\
$\sigma^*R=R$.

Now we relate the root space decompositions of
$\ggg^\sigma$ and $\ggg$. We define the subsystem of compact roots
$R_\bullet$ by
$$
R_\bullet =\{\alpha\in R\,\,\vert\,\, \sigma^*\alpha = -\alpha \}=
\{\alpha\,\,\vert\,\,\alpha(\gh^-)=0\}
$$
and denote by $R'=R\setminus R_\bullet$ the complementary set of
non-compact roots. We can choose a system $\Pi$ of simple roots of
$R$ such that the corresponding system $R^+$ of positive roots
satisfies the condition:
$$
R'_+:=R'\cap R^+\,\,\,\hbox{\rm is} \,\,\,\,\sigma\hbox{\rm -invariant}.
$$
In this case, $\Pi$ is called a $\sigma$-{\bf fundamental system}
of roots.
\\
We denote by $\Pi_\bullet =\Pi\cap R_\bullet$ the set of
compact simple roots (which are also called black) and by $\Pi'
=\Pi\setminus \Pi_\bullet$ the non-compact simple roots (called
white). The action of $\sigma^*$ on white roots satisfies the
following property:
\\
for any $\alpha\in\Pi'$ there exists a unique $\alpha'\in\Pi'$
such that\ $\sigma^*\alpha-\alpha'$\ is a linear combination of
black roots.
In this case, we say that $\alpha, \,\alpha'$ are
$\sigma$-{\bf equivalent}.

The information about the fundamental system ($\Pi =\Pi_\bullet \cup
\Pi'$) together with the $\sigma$-equivalence can be visualized in
terms of the {\bf Satake diagram}, which is defined as follows:
\newline
on the Dynkin diagram $\G$ of the system of simple roots $\Pi$,
we paint the vertices which correspond to black roots black and
we join the vertices which correspond to
$\sigma$-equivalent roots $\alpha,\,\alpha'$ by a curved arrow.
We recall that there is a
natural $1-1$ correspondence between Satake diagrams subordinated to
the Dynkin diagram of a complex semisimple Lie algebra $\ggg$, up to
isomorphisms, and real forms $\ggg^\sigma$ of $\ggg$,
up to conjugations.
The list of Satake diagram of real
simple Lie algebras is known (see e.g. \cite{GOV}).
\smallskip\par
The following proposition describes fundamental gradations of
semisimple complex (respectively, real) Lie algebras in terms of crossed Dynkin (respectively, Satake) diagrams (see
e.g. \cite{Dj}, \cite{AMT}).
\begin{prop}\label{prop-grad-real} A fundamental gradation of a
complex semisimple Lie algebra $\ggg=\sum\ggg_p$ can be given by a crossed Dynkin diagram $\G$. Crossed
nodes belong to a subset $\Pi^1$ of a simple root system $\Pi =\Pi^0\cup\Pi^1$.
The corresponding grading vector ${d}\in\ga$ is given by
$$
\alpha_i({d})=
\begin{cases}
0 & \hbox{\rm if}\,\, \alpha_i\in\Pi^0\\
1 &\hbox{\rm if}\,\, \alpha_i\in\Pi^1\,.
\end{cases}
$$
The subspaces $\ggg_p$ are given by
$$
\ggg_p =\sum_{\alpha({d})=p}\ggg_\alpha\,.
$$A real form $\ggg^\s$ is consistent with the gradation (i.e.\ $
d\in\ggg^\s$) if and only if the corresponding Satake diagram
$\hat\G$ satisfies the following properties:
\begin{enumerate}
\item[i)] all black nodes of $\hat\G$ are uncrossed; \\
\item[ii)] two
nodes related by an arrow are both crossed or uncrossed.
\end{enumerate}
\end{prop}

\subsection{Computation of the Koszul form and the main theorem.}
Now we compute the Koszul form of a homogeneous para-complex manifold
$(M=G^\sigma/H^\sigma,K_M,\vol)$, where $G^\s$ is a real form of a complex semisimple Lie group $G$,
 $M = G^\s/H^\s$ is a covering of a semisimple adjoint orbit $\Ad_{G^\s}d$, $K_M$ is the invariant
para-complex structure on $M$ defined by the gradation of the Lie algebra $\ggg^\sigma$ with the grading element $d \in \ggg^\s$
and $\vol$ is an invariant volume form on $M$. According to Proposition
\ref{propkoszul}, it is sufficient to describe the Koszul form $\psi$
on the graded Lie algebra $\ggg^\s$ or its complexification $\ggg$.

We choose a Cartan subalgebra $\ga\subset\ggg_0$ of the Lie algebra $\ggg$ and denote by $R$ the root system of $(\ggg,\ga)$.
Let
$$
\Pi =\Pi^0\cup \Pi^1
$$
be the decomposition of a simple root system $\Pi$ of the root system $R$ which corresponds to the gradation and by
$$
P = P^0\cup P^1
$$
the corresponding decomposition of the fundamental weights.\newline
We denote by $R^+$ the set of positive roots with respect to the basis $\Pi$ and set
$$
R^+_0=\{\alpha\in R^+\,\,\vert\,\,\ggg_\alpha\subset \ggg_0\}\,.
$$
The following lemma describes the Koszul form $\psi \in \ga^\s \subset \ggg^\s \subset \ggg$ defined by \eqref{koszulform1} in
terms of fundamental weights.
\begin{lemma}\label{psilemma}
The Koszul $1$-form $\psi \in \ga^*$ is equal to
$$
\psi =2(\delta^\ggg -\delta^\gh)
$$
where
$$
\delta^\ggg=\sum_{\alpha\in R^+}\alpha\,,\qquad
\delta^\gh=\sum_{\alpha\in R_0^+}\alpha\,,
$$
and the linear forms on the Cartan
subalgebra $\ga$ are considered as linear forms on $\ggg$ which
vanish on root spaces $\ggg_\alpha$.
\end{lemma}
\begin{proof} First of all, remark that for any $E_\alpha\in\ggg_\alpha$ we have
$$
KE_\alpha =
\left\{
\begin{array}{rl}
0 & \mbox{\rm if}\,\, \pm\alpha\in R^+_0\\[3pt]
\pm E_\alpha &\mbox{\rm if}\,\, \pm\alpha\in R^+\setminus R^+_0\,.
\end{array}
\right.
$$
Hence the endomorphisms $K\ad_{E_\alpha}$ and $\ad_{KE_\alpha}$ are nilpotent and consequently $\psi(E_\alpha)=0$, for
any root $\alpha\in R$. In particular $\psi\vert_\gm =0$.\newline
Assume now that $X=t$ belongs to the Cartan subalgebra $\ga$. Then $\ad_t(E_\alpha)=\alpha(t) E_\alpha$.
Therefore
\begin{eqnarray*}
\psi(t)&=&\tr_\gm (K\ad_t -\ad_{Kt})=\tr_{\gm}(K\ad_t)=\tr_{\gm_+}(K\ad_t)+\tr_{\gm_-}(K\ad_t)\\[3pt]
&=& \sum_{\alpha\in R^+\setminus R^+_0}\alpha(t)-\sum_{-\alpha\in R^+\setminus R^+_0}\alpha(t) =2\sum_{\alpha\in R^+\setminus R^+_0}\alpha(t)\\[3pt]
&=& 2\sum_{\alpha\in R_+}\alpha(t)-2\sum_{\alpha\in R^0_+}\alpha(t)\,.
\end{eqnarray*}
\end{proof}
By the last lemma and Proposition 4.1 of \cite{AP}, it follows
\begin{prop}\label{perelomov} Let $\Pi=\Pi^0\cup \Pi^1=\{\alpha_1,\ldots ,\alpha_\ell\}$
be the simple root system (corresponding to the gradation) and
denote by $\pi_i$ the fundamental weight corresponding
to the simple root $\alpha_i$, namely
$$
2\frac{(\pi_i,\alpha_j)}{(\alpha_j,\alpha_j)} =\delta_{ij}\,.
$$
If $P^1 =\{\pi_{i_1},\ldots , \pi_{i_r}\}$, then the Koszul form $\psi$ is equal to
\begin{equation}\label{koszulform}
\psi =2\sum_{\pi\in
P^1}n_\pi\pi=2\sum_{h=1}^ra_{i_h}\pi_{i_h}\,,
\end{equation}
where
\begin{equation}\label{koszulcoefficients}
a_{i_h}=2+b_{i_h}\,,\quad \hbox{with}\quad b_{i_h}=-2\frac{(\delta^\gh,\alpha_{i_h})}{(\alpha_{i_h},\alpha_{i_h})}\geq 0\,.
\end{equation}
\end{prop}
\smallskip
Note that the $1$-form $\psi$ depends only on the decomposition
$\Pi=\Pi^0\cup\Pi^1$.\smallskip\par Let us denote by
$\{X_\alpha\,,\alpha\in R,\,\,H_i\,,\,\,i=1,\ldots,\ell\}$
a Chevalley basis of the Lie algebra $\ggg$. For $\alpha\in R$, we
denote by $\omega^\alpha$ the linear forms on $\ggg$ such that
$$
\omega^\alpha (\ga)=0\,,\quad\omega^\alpha(X_\beta)=\delta^\alpha_\beta\,,
$$
for any $\beta\in R$. If $\xi\in \ga^*$ and
$\alpha\in R$, then we put
$$
n(\xi,\alpha) =2\frac{(\xi,\alpha)}{(\alpha,\alpha)}\,.
$$
The next lemma easily follows from the commutation rules in the Lie algebra $\ggg$ (see \cite[p. 145]{H} and also \cite{AP}).
\begin{lemma}\label{dxilemma}
The differential of any $1$-form $\xi\in \ga^*$ is given by
$$
d\xi =\sum_{\alpha\in R^+} n(\xi,\alpha) \omega^\alpha \wedge \omega^{-\alpha}\,.
$$
Moreover $d\xi$ is an $\ad_\ga$-invariant $2$-form on $\ggg$ with kernel
$$
\ker (d\xi) = \ga + \mbox{\rm span}\left\{E_\alpha\,\,\vert\,\, (\xi,\alpha)=0 \right\}\,.
$$
\end{lemma}
By Proposition \ref{perelomov} and Lemma \ref{dxilemma} we obtain the following
\begin{cor}\label{corkoszul}
The $\ad_\gh$-invariant $2$-form $\rho = d\psi$ on the Lie algebra
$\ggg$ has kernel $\gh$.
\end{cor}
\begin{proof} Let $P^1= \{\pi_{i_1},\ldots ,\pi_{i_r}\}$ be fundamental weights which correspond to crossed
simple roots $\Pi^1=\{\alpha_{i_1},\ldots ,\alpha_{i_r}\}$. Then $\psi = 2(a_{i_1}\pi_{i_1}+\cdots + a_{i_r}\pi_{i_r})$. By Lemma
\ref{dxilemma}, in order to
determine the kernel of $\psi$, it is sufficient to describe all roots $\alpha$ with scalar product $(\psi,\alpha)=0$. Recall that
$$
(\pi_{i},\alpha_{i})=\frac{1}{2}(\alpha_{i},\,\alpha_{i})
$$
and the scalar product of $\pi_i$ with the other simple roots is
zero. We can write any root $\alpha$ as
$$
\alpha = k_{1}\alpha_{1}+\cdots + k_{r}\alpha_{i_r} +\beta
$$
where $\beta$ is a linear combination of other simple roots. Then
\begin{eqnarray*}
(\psi,\alpha) &=&2\,(a_{i_1}\pi_{i_1}+\cdots + a_{i_r}\pi_{i_r},\,k_{1}\alpha_{1}+\cdots + k_{r}\alpha_{i_r}+\beta) =\\
&=&\left(k_1a_{i_1} \left((\alpha_{i_1},\alpha_{i_1}\right)+\cdots +k_ra_{i_r} \left(\alpha_{i_r},\alpha_{i_r}\right)\right)\,.
\end{eqnarray*}
This shows that $(\psi,\alpha)=0$ if and only if $k_1=\cdots
=k_r=0$, that is $\alpha\in R^0$. Hence the kernel of $\rho$ is
$\gh$.
\end{proof}
By Corollary \ref{corhomogparakaehler} and Corollary \ref{corkoszul} and the above discussions, we obtain the following
\begin{thm} Let $R$ be a root system of a complex semisimple Lie algebra $\ggg$ with respect to a Cartan subalgebra $\ga$ and
$\ggg =\ggg_{-k}+\cdots + \ggg_k$ the fundamental gradation with
the graded element $d$ associated with a decomposition
$\Pi=\Pi^0\cup\Pi^1$ of a simple root system $\Pi\subset R$. Let
$\sigma$ be an admissible anti-involution of $\ggg$ which defines
the graded real form $\ggg^\sigma$ of $\ggg$, \, $G^\s$  a connected real
semisimple Lie group with the Lie algebra $\ggg^\s$ and
$M=G^\sigma/H^\s $ a covering of a semisimple adjoint orbit $
Ad_{G^\sigma}(d)$.
 Denote by $K$ and $\psi$ the invariant para-complex structure on $M$ and the Koszul form associated with the
 gradation of $\ggg^\sigma$ and by $\rho$ the invariant symplectic form on $M$ defined by $d\psi$.\\
 Then, for any $\lambda\neq 0$, the pair $(K,\lambda\rho)$
is an invariant para-K\"ahler Einstein structure on $M$ and this construction exhausts all
homogeneous para-K\"ahler Einstein manifolds of real semisimple Lie groups.
\end{thm}

\subsection{Examples.}
In this subsection we describe the Koszul form $\psi$ for adjoint
orbits $M=G^\sigma/H=Ad_{G^\sigma}(h)$ of some simple Lie group
$G^\sigma$. In the case of $G_2$ we also indicate the para-K\"ahler
Einstein form $\rho =d\psi$.\vskip.2truecm\noindent
\paragraph{\bf Case  $\ggg = A_\ell = \mathfrak{sl}({n+1}, \mathbb{C})$.}

The root system is $R = \{ \epsilon_i - \epsilon_j\}$ and the  system
of simple roots is $\Pi = \{\alpha_i = \epsilon_i - \epsilon_{i+1} \}$. We will consider two real form of $\ggg$.
\smallskip
\\
i) $G^\s = \SL(\ell +1,\R)$.\\
  We denote by $M = \SL(\ell+1)/H$ the
homogeneous manifold  associated with the
subsystem of simple roots given by \\
$\Pi^1=\{\alpha_{i_1},\ldots
,\alpha_{i_r}\,\,\vert\,\, 1\leq i_1<i_2<\cdots < i_r\leq\ell\}$.\\
The manifold $M$ has dimension $ (\ell
+1)^2-\displaystyle\sum_{k=1}^{r+1}(i_{k}-i_{k-1})^2\,, $ where we
assume that $i_0=0$ and $i_{r+1}=\ell +1$. Taking into account
formula \eqref{koszulform}, we get the Koszul form
$$
\psi = 2\sum_{k=1}^r(i_{k+1}-i_{k-1})\pi_{i_k}\,.
$$
ii) $ G^\s= \SL(2,\quater)$ and
$\Pi^1=\{\alpha_2\}$. \\
This case corresponds to the crossed Satake
diagram
$$
\xymatrix @M=0pt @R=2pt @!C=6pt{
 \bullet \ar@{-}[r] &\APLcirc{\times} \ar@{-}[r] &\bullet }
$$
The manifold $M$ has dimension $8$. According to \eqref{koszulcoefficients}, we have $b_2=2$ and therefore
$a_2 =4$. Hence the Koszul form is
$$
\psi =2a_2\pi_2= 8\pi_2=4\alpha_1+8\alpha_2+4\alpha_3 \,.
$$
Taking into account lemma \ref{dxilemma}, by a direct computation we get
$$
\begin{array}{lll}
\rho &=& 8\left(\omega^{\alpha_2}\wedge\omega^{-\alpha_2}+\omega^{\alpha_1+\alpha_2}\wedge\omega^{-(\alpha_1+\alpha_2)}+
\omega^{\alpha_2+\alpha_3}\wedge\omega^{-(\alpha_2+\alpha_3)}+\right. \\[6pt]
&{}&+ \left.\omega^{\alpha_1+\alpha_2+\alpha_3}\wedge\omega^{-(\alpha_1+\alpha_2+\alpha_3)}
\right)\,.
\end{array}
$$

\paragraph{\bf  Case  of the  complex exceptional Lie algebra \bf $\ggg = \ggg_2 $}

The system of  simple roots can be  written as $\Pi = \{ \alpha_1, \alpha_2 \}$ and the associated system of positive roots  is
$$
R^+=\{\alpha_1,\alpha_2,\alpha_1+\alpha_2,2\alpha_1+\alpha_2,
3\alpha_1+\alpha_2,3\alpha_1+2\alpha_2\}.
$$
We consider the normal real form  $ (\ggg_2)^\s$ of $\ggg_2$  and  denote by
 $G^\s = G_2$   the corresponding  simple Lie group.
\newline The fundamental weights are
$$
\pi_1 =2\alpha_1 +\alpha_2\,, \quad
\pi_2 =3\alpha_1 +2\alpha_2\,.
$$
We have the following three cases:
\begin{itemize}
\item[i)]\
$
\xymatrix @M=0pt @R=2pt @!C=6pt{
{\APLcirc{\times}}\ar@{=}[r]|{\SelectTips{cm}{}\dir{>}}\ar@{-}[r]&{\circ}\\
{\alpha_1}&{\alpha_2} }\qquad
$
$\Pi^1=\{\alpha_1\}\,$,\vskip.3truecm\noindent
\item[ii)]\
$
\xymatrix @M=0pt @R=2pt @!C=6pt{
{\circ}\ar@{=}[r]|{\SelectTips{cm}{}\dir{>}}\ar@{-}[r]&{\APLcirc{\times}}\\
{\alpha_1}&{\alpha_2} }\qquad
$
$\Pi^1=\{\alpha_2\}\,$,\vskip.3truecm\noindent
\item[iii)]\
$
\xymatrix @M=0pt @R=2pt @!C=6pt{
{\APLcirc{\times}}\ar@{=}[r]|{\SelectTips{cm}{}\dir{>}}\ar@{-}[r]&{\APLcirc
{\times}}\\
{\alpha_1}&{\alpha_2} }\qquad
$
$\Pi^1=\{\alpha_1,\alpha_2\}\,$.\vskip.3truecm\noindent
\end{itemize}
i) The manifold $M$ has dimension $10$. By applying formula \eqref{koszulcoefficients}, we obtain
$b_1 =3$ and consequently $a_1=5$. Therefore, by formula \eqref{koszulform},
the Koszul form $\psi$ can be expressed as
$$
\psi = 10 \pi_1=10\left(2\alpha_1+\alpha_2\right)\,.
$$
By Lemma \ref{dxilemma}, it follows that the para-K\"ahler Einstein form $\rho$ is given by
$$
\begin{array}{lll}
\rho &=& 10\left(\omega^{\alpha_1}\wedge\omega^{-\alpha_1}+\omega^{\alpha_1+\alpha_2}\wedge\omega^{-(\alpha_1+\alpha_2)}+
2\,\omega^{2\alpha_1+\alpha_2}\wedge\omega^{-(2\alpha_1+\alpha_2)}+\right. \\[6pt]
&{}& +\left.\omega^{3\alpha_1+\alpha_2}\wedge\omega^{-(3\alpha_1+\alpha_2)}+\omega^{3\alpha_1+2\alpha_2}\wedge\omega^{-(3\alpha_1+2\alpha_2)}
\right)\,.
\end{array}
$$
ii) The manifold $M$ has dimension $10$. By \eqref{koszulcoefficients}, we get
$b_2=1$ and consequently $a_2=3$. Hence, according
to \eqref{koszulform} and Lemma \ref{dxilemma}, the Koszul form $\psi$ and the para-K\"ahler Einstein form $\rho$ are
given respectively by
$$
\psi = 6 \pi_2 =6\left(3\alpha_1+2\alpha_2\right)
$$
and
$$
\begin{array}{lll}
\rho &=& 6\left(\omega^{\alpha_2}\wedge\omega^{-\alpha_2}+3\,\omega^{\alpha_1+\alpha_2}\wedge\omega^{-(\alpha_1+\alpha_2)}+
3\,\omega^{2\alpha_1+\alpha_2}\wedge\omega^{-(2\alpha_1+\alpha_2)}+\right. \\[6pt]
&{}& +\left.\omega^{3\alpha_1+\alpha_2}\wedge\omega^{-(3\alpha_1+\alpha_2)}+2\,\omega^{3\alpha_1+2\alpha_2}\wedge\omega^{-(3\alpha_1+2\alpha_2)}
\right)\,.
\end{array}
$$
iii) The manifold $M$ has dimension $12$. In this case we get
$$
\psi = 4(\pi_1+\pi_2)=4(5\alpha_1 +3\alpha_2)
$$
and
$$
\begin{array}{lll}
\rho &=& 4 \left(\omega^{\alpha_1}\wedge\omega^{-\alpha_1}+ \omega^{\alpha_2}\wedge\omega^{-\alpha_2}+
4\,\omega^{\alpha_1+\alpha_2}\wedge\omega^{-(\alpha_1+\alpha_2)}+\right.\\[6pt]
&{}&+\left. 5\,\omega^{2\alpha_1+\alpha_2}\wedge\omega^{-(2\alpha_1+\alpha_2)}+
2\,\omega^{3\alpha_1+\alpha_2}\wedge\omega^{-(3\alpha_1+\alpha_2)}+\right.\\[6pt]
&{}&+\left.3\,\omega^{3\alpha_1+2\alpha_2}\wedge\omega^{-(3\alpha_1+2\alpha_2)}
\right)\,.
\end{array}
$$

\end{document}